\documentclass[3p]{elsarticle}
\usepackage{amsmath}
\usepackage{amssymb}
\usepackage{float}
\usepackage{amsfonts}
\usepackage{mathrsfs}
\usepackage{tikz}
\usepackage{tikz-cd}
\usepackage{booktabs}
\usepackage{multirow}
\usepackage{siunitx}
\usepackage{doi}
\usepackage{listings}
\usepackage{alltt}
\usepackage{tabularx}
\usepackage[font=small,labelfont=bf, hypcap=false]{caption}
\usepackage{amsthm}
\usepackage{hyperref}
\usepackage{enumitem}
\usepackage{graphicx}
\usepackage[title]{appendix}
\usepackage{xcolor}
\usepackage{textcomp}
\usepackage{manyfoot}
\usepackage{algorithm}
\usepackage{algorithmicx}
\usepackage{algpseudocode}
\usepackage{bookmark}
\usepackage{anyfontsize}
\usepackage{rotating}
\usepackage[utf8]{inputenc}
\usepackage[english]{babel}
\usepackage{array}
 \lstdefinestyle{pythonstyle}{ language=Python,
  basicstyle=\ttfamily\footnotesize,
  keywordstyle=\color[HTML]{185FA5}\bfseries,
  stringstyle=\color[HTML]{3B6D11},
  commentstyle=\color[HTML]{888780}\itshape,
  numberstyle=\tiny\color[HTML]{888780},
  numbers=left,
  numbersep=6pt,
  stepnumber=1,
  frame=single,
  framerule=0.4pt,
  rulecolor=\color[HTML]{D3D1C7},
  backgroundcolor=\color[HTML]{F1EFE8},
  breaklines=true,
  breakatwhitespace=true,
  tabsize=4,
  showstringspaces=false,
  captionpos=b,
}

\usetikzlibrary{arrows}

\usetikzlibrary{matrix, backgrounds}
\newcommand{\0}{\scalebox{2}{$0$}}

\newtheorem{theorem}{Theorem}
\newtheorem{lemma}{Lemma}
\newtheorem{corollary}{Corollary}
\newtheorem{proposition}{Proposition}
\newtheorem{definition}{Definition}
\newtheorem{example}{Example}
\newtheorem{remark}{Remark}

\newcommand{\E}{\mathbb{E}^4_1}
\newcommand{\R}{\mathbb{R}}
\newcommand{\ip}[2]{\langle #1,#2\rangle}
\newcommand{\Ax}{\mathcal{A}}
\newcommand{\sgn}{\operatorname{sign}}
\newcommand{\spn}{\operatorname{span}}

\begin{document}

\begin{frontmatter}

    \title{Axes of null Cartan normal helices with applications to isophotic and silhouette curves in Minkowski space-time}
    
        \author[Umut]{Derya Sağlam}
        \author[Umut]{Umut Selvi \corref{cor1}}
        \ead{umut.selvi@hbv.edu.tr}
            \cortext[cor1]{Corresponding author}
        \author[Sevgi]{Sevgi Esen Almalı}
        \ead{sevgi_esen@hotmail.com}

    \address[Umut]{Department of Mathematics, Ankara Hacı Bayram Veli University, Ankara, Türkiye}
    \address[Sevgi]{Deparment of Mathematics, Kırıkkale University, Kırıkkale, Türkiye}
\begin{abstract}
We study null Cartan curves in Minkowski space-time $\E$ that are normal helices with respect to C-constant normal fields $V=\lambda T+\mu N+\nu B_2$, whose components in the Cartan frame are constant. If $\kappa_3$ has no zeros, the constant vectors $W$ for which $\ip{V}{W}$ is constant correspond to common solutions of two linear differential equations, of orders four and at most two. Hence they form a space of dimension at most two, $B_2$ is never an axis field, and the normal helices with any other axis field $V$ and prescribed $\kappa_3$ are obtained by two quadratures, explicitly for $V=\lambda T\pm B_2$. For null Cartan helices, $V$ admits an axis if and only if one of three explicit linear conditions on $(\lambda,\mu,\nu)$ holds; then $\ip{V}{W}=0$, and the axes are spanned by the invariant null lines and spacelike plane of the boost--rotation group generating the helix. For fields with $\mu\nu\neq0$ we obtain a Riccati criterion and determine all fields moving in a fixed plane; for fields $\lambda T+N$ we find non-helical curves with two linearly independent axes, unlike in dimension three. Finally, we introduce a normalized Darboux frame for null curves on timelike hypersurfaces and characterize the isophotic and silhouette curves with C-constant normal.
\end{abstract}    
    

\end{frontmatter}

\section{Introduction}

The notion of a helix occupies a central place in the differential geometry of curves. A unit-speed curve in $\mathbb{E}^3$ is a general helix, that is, $\ip{T}{\mathbf a}$ is constant for a fixed unit vector $\mathbf a$, if and only if the ratio of its torsion to its curvature is constant (Lancret's theorem). Lucas and Ortega-Yag\"ues \cite{lucas} generalized this notion by means of F-constant vector fields: a vector field $V$ along a curve is F-constant if its components with respect to the Frenet frame are constant, and the curve is a normal, rectifying or osculating helix if such a field $V$, lying in the normal, rectifying or osculating plane, satisfies $\ip{V}{\mathbf a}=\mathrm{const}$ for a fixed vector $\mathbf a$. Helices with F-constant vector fields in $\mathbb{E}^4$ were investigated in \cite{saglam26}, and timelike helices with F-constant vector fields in Minkowski space $\mathbb{E}^3_1$ in \cite{saglambada}.

For null (lightlike) curves the situation is different. Since $\ip{\alpha'}{\alpha'}=0$, there is no arc length parameter; instead one uses the pseudo-arc parameter and the Cartan frame \cite{bonnor,duggalbej,duggaljin,walrave}. Null curves are the worldlines of massless particles, and the pseudo-arc length of a null curve has been used as the action of a relativistic particle model \cite{nersessian}. Null helices in Lorentzian space forms were studied in \cite{fgl01} and null generalized helices in \cite{fgl02}; null helices and Cartan slant helices in $\mathbb{E}^3_1$ were considered in \cite{liupei19}, $k$-type null slant helices in \cite{nes15,nes18}, slant helices with respect to the Bishop frame in \cite{gur}, and null Cartan Bertrand curves in Minkowski space-time in \cite{sunpei}. For further recent work on null curves and on lightcone framed curves see \cite{djnes20,gaber,liupei20,pei25}. In $\mathbb{E}^3_1$, Ne\v{s}ovi\'c \cite{nes25} introduced null Cartan normal helices, whose axis field is $\lambda_0T+N$ with a constant $\lambda_0\neq0$, determined their torsion and proved that null Cartan helices are the only such curves with two axes; null Cartan normal isophotic and normal silhouette curves on timelike surfaces in $\mathbb{E}^3_1$ were studied in \cite{djnes24}, and isophote curves on spacelike surfaces in $\mathbb{E}^3_1$ in \cite{doganyayli}.

In this paper we study the corresponding problem in Minkowski space-time. Let $\alpha$ be a null Cartan curve in $\E$ with Cartan frame $\{T,N,B_1,B_2\}$ and curvatures $\kappa_2,\kappa_3$. A \emph{C-constant normal field} is a field $V=\lambda T+\mu N+\nu B_2$ with constant coefficients, and $\alpha$ is a \emph{normal helix with axis field $V$} if $\ip{V}{W}$ is constant for some constant vector $W\neq0$. We use the standard Cartan frame, which exists and is unique for every null Cartan curve (Theorem~\ref{thm:frame}); hence these notions are invariant under isometries, and the fields of type~III, $\lambda=0$, are precisely the C-constant sections of the screen distribution of $\alpha$ (Remark~\ref{rem:natural}). One may also try to normalize the frame so that $N'=\tilde k_2\tilde B_1-\tilde B_2$ with $\tilde B_1$ spacelike and $\tilde B_2$ null; as shown in Remark~\ref{rem:other}, such a frame exists only when $\ip{\alpha'''}{\alpha'''}\geq0$ and its curvature $\tilde k_2$ is determined only up to sign. In $\E$ the C-constant normal fields form, up to scaling, a two-parameter family, which contains the fields $T$, $N$ and $B_2$ of the $k$-type null slant helices of \cite{nes15}.

Fields with constant components are natural for both geometric and physical reasons. The normal space $T^{\perp}=\spn\{T,N,B_2\}$ of a null curve contains the tangent, so a normal field $V=\lambda T+\mu N+\nu B_2$ splits into its screen part $\mu N+\nu B_2$, which carries the norm $\ip{V}{V}=\mu^2+\nu^2$, and its null part $\lambda T$, which does not affect the norm but adds $\lambda\ip{T}{W}$ to $\ip{V}{W}$; for $V=T$ the condition that $\ip{V}{W}$ be constant defines the null slant helices of type $k=0$ \cite{nes15}. When a C-constant field $V$ is the unit normal of a timelike hypersurface along the curve and $\nu\neq0$, its $T$-component is $\lambda=-\nu\tau$, where $\tau$ is the analogue of the geodesic torsion in the normalized Darboux frame of Section~\ref{sec:hyper}; thus $\lambda$ measures how the hypersurface twists around the curve. Physically, a null curve is the worldline of a massless particle. In the particle model of Nersessian and Ramos \cite{nersessian}, whose action is proportional to the pseudo-arc length, the classical trajectories of positive mass are the null curves with $\kappa_3\equiv0$ and constant $\kappa_2<0$, and along them $\lambda$ and $\nu$ are, up to constant factors, the projections of $V$ on the conserved momentum and on the Pauli--Lubanski vector. When $\kappa_3$ has no zeros these conserved directions are no longer available, and for a null Cartan helix their role is taken by the invariant subspaces of the boost--rotation group of which the helix is an orbit; see the discussion after Theorem~\ref{thm:main}.

Our first main result reduces the problem, for arbitrary $\kappa_2$ and nowhere vanishing $\kappa_3$, to linear differential equations (Theorems~\ref{thm:reduction} and~\ref{thm:construct}). The map $W\mapsto\ip{W}{T}$ identifies the constant vectors $W$ for which $\ip{V}{W}$ is constant with the common solutions of the fourth-order linear equation satisfied by the components of $T$ and of the linear equation $\mu f''+\lambda f'-\nu\kappa_3f=0$, of order at most two. Consequently such vectors form a space of dimension at most two and $B_2$ is never an axis field (both statements fail if $\kappa_3\equiv0$, when $B_2$ is constant), and for every $V$ that is not a multiple of $B_2$ and every prescribed nowhere vanishing $\kappa_3$ the normal helices with axis field $V$ are obtained, away from the zeros of $\ip{W}{T}$, by two quadratures. For the fields $V=\lambda T\pm B_2$ of type~II with $\lambda\neq0$ the solution is explicit (Theorem~\ref{thm:typeII}): $\alpha$ is a normal helix with axis field $V$ if and only if $\kappa_2$ belongs to an explicit two-parameter family determined by $\kappa_3$, and one of the parameters is proportional to $\ip{V}{W}$.

Our second main result concerns null Cartan helices, that is, null Cartan curves with constant $\kappa_2$ and constant $\kappa_3\neq0$ (Theorem~\ref{thm:main}, Corollaries~\ref{cor:unit} and~\ref{cor:types}). For these curves the axes correspond to the common roots of the polynomials $\mu\theta^2+\lambda\theta-\nu\kappa_3$ and $\theta^4-2\kappa_2\theta^2-\kappa_3^2$. Hence every axis satisfies $\ip{V}{W}=0$, and $\alpha$ is a normal helix with axis field $V$ if and only if one of the conditions $\Lambda_+=0$, $\Lambda_-=0$, $\lambda=\Lambda_0=0$ holds, where $\Lambda_\pm,\Lambda_0$ are explicit linear forms in $(\lambda,\mu,\nu)$. The constant vectors $W$ with $\ip{V}{W}=0$ then form the span of one or two fixed null lines if $\Lambda_+=0$ or $\Lambda_-=0$, and a fixed spacelike plane if $\lambda=\Lambda_0=0$; these subspaces are determined by the helix alone (Proposition~\ref{prop:decomp}).

We then study fields with $\mu\nu\neq0$ and fields of type~I (Theorems~\ref{thm:munu}, \ref{thm:planar} and~\ref{thm:typeI}). For fields with $\mu\nu\neq0$, which include all fields of type~III except the multiples of $N$ and $B_2$, the existence of an axis with $\ip{V}{W}=0$ reduces, away from the zeros of an explicit function $g_1$, to a Riccati equation built from $\kappa_2$ and $\kappa_3$. Moreover, the constant vectors $W$ with $\ip{V}{W}=0$ form a two-dimensional space exactly when $V$ moves in a fixed plane through the origin; besides null Cartan helices, such fields occur along non-helical curves with $\kappa_3$ an exponential function. For $V=\lambda T+N$ with $\lambda\neq0$ the condition for an axis with $\ip{V}{W}=0$ becomes a Bernoulli equation for $\kappa_3$ when $\kappa_2\equiv0$, nonzero constants $\ip{V}{W}$ occur, and there are non-helical curves with two linearly independent axes for the same field (Corollary~\ref{cor:typeI}), in contrast with the three-dimensional case \cite{nes25}.

The methods of the paper are elementary, namely linear ordinary differential equations and linear algebra, and apart from the fields of type~II and the field $T$ the results for arbitrary curvatures are criteria rather than closed formulas. We see the main contribution of the paper in two places. The first is a unifying framework: a single pair of linear equations treats all C-constant normal fields and all curvatures at once, contains the $k$-type null slant helices of \cite{nes15} for $k=0,1,3$, and reduces for null Cartan helices to the common roots of two polynomials. The second is a set of results and examples that separate the situation in $\E$ from that in $\mathbb{E}^3_1$, namely the boost--rotation structure of null Cartan helices and the non-helical curves with two linearly independent axes, and that show which statements for helices fail for non-constant curvatures (Examples~\ref{ex:56} and~\ref{ex:58}). Table~\ref{tab:summary} at the end of Section~\ref{sec:main} summarizes the results for the three types of fields.

Finally, in Section~\ref{sec:hyper} we introduce a normalized Darboux frame for a null curve on a timelike hypersurface, determined by four functions, and describe the null Cartan curves that are isophotic or silhouette curves when the unit normal along the curve is C-constant (Theorem~\ref{thm:iso}). By Proposition~\ref{prop:darboux}, this happens exactly when two of the four functions are constant, and the Cartan curvatures are then explicit functions of the surface invariants. For instance, an asymptotic null curve of this kind is isophotic if and only if its normal is not $\pm B_2$ and $\kappa_2$ has the explicit form found for type~II, and an asymptotic null Cartan helix is a silhouette if and only if the $T$-component of the normal equals $\pm\omega$, where $\pm i\omega$ are the imaginary roots of $\theta^4-2\kappa_2\theta^2-\kappa_3^2$, that is, if and only if the surface invariant $\tau$ satisfies $|\tau|=\omega$ (Corollary~\ref{cor:iso}).

All results are illustrated by explicit examples in Section~\ref{sec:examples}. Every lemma, proposition, theorem, corollary and example has been checked symbolically or numerically with the accompanying Python package \cite{code}, which also generates the figures.

Throughout, $\E=(\R^4,\ip{\cdot}{\cdot})$ with $\ip{x}{y}=-x_1y_1+x_2y_2+x_3y_3+x_4y_4$ (see \cite{oneill}). A vector $v\neq0$ is spacelike, timelike or null according as $\ip{v}{v}>0$, $<0$ or $=0$. We write $\det(v_1,v_2,v_3,v_4)$ for the determinant of the matrix with columns $v_1,\dots,v_4$; a basis is positively oriented if this determinant is positive. We denote by $e_1,\dots,e_4$ the standard basis, and primes for derivatives with respect to $s$. All functions are smooth and $I$ denotes an open interval.

\section{Null Cartan Curves and C-Constant Normal Fields}\label{sec:prelim}

\subsection{The Cartan frame}
Let $\alpha:I\to\E$ be a null curve, $\ip{\alpha'}{\alpha'}=0$. Then $\ip{\alpha'}{\alpha''}=0$, so $\alpha''$ is spacelike or null; if $\alpha''$ is null on an open set, it is parallel to $\alpha'$ there and $\alpha$ is a reparametrized null straight line. We therefore assume $\ip{\alpha''}{\alpha''}>0$ and use the \emph{pseudo-arc parameter} $s=\int\ip{\alpha''(t)}{\alpha''(t)}^{1/4}\,dt$, for which
\begin{equation}\label{eq:pa}
\ip{\alpha'}{\alpha'}=0,\qquad \ip{\alpha''}{\alpha''}=1
\end{equation}
(see \cite{bonnor,duggalbej,duggaljin}). From now on $\alpha$ is parametrized in this way.

\begin{lemma}\label{lem:pa}
Under \eqref{eq:pa},
$\ip{\alpha'}{\alpha''}=0$, $\ip{\alpha'}{\alpha'''}=-1$, $\ip{\alpha''}{\alpha'''}=0$, $\ip{\alpha'}{\alpha^{(4)}}=0$ and $\ip{\alpha''}{\alpha^{(4)}}=-\ip{\alpha'''}{\alpha'''}$.
\end{lemma}

\begin{proof}
Differentiate $\ip{\alpha'}{\alpha'}=0$ and $\ip{\alpha''}{\alpha''}=1$ to get the first and third identities; differentiating $\ip{\alpha'}{\alpha''}=0$, $\ip{\alpha'}{\alpha'''}=-1$ and $\ip{\alpha''}{\alpha'''}=0$ gives the remaining ones.
\end{proof}

\begin{theorem}\label{thm:frame}
Let $\alpha$ be a null curve parametrized by pseudo-arc. Put $T=\alpha'$, $N=\alpha''$,
\begin{equation}\label{eq:defframe}
\kappa_2=-\tfrac12\ip{\alpha'''}{\alpha'''},\qquad B_1=\kappa_2T-\alpha''',
\end{equation}
and let $B_2$ be the vector field determined by
\[
\ip{B_2}{X}=\det(T,N,B_1,X)\qquad\text{for all } X\in\E .
\]
Then:
\begin{enumerate}
\item $\ip{T}{T}=\ip{B_1}{B_1}=0$, $\ip{T}{B_1}=\ip{N}{N}=\ip{B_2}{B_2}=1$, all other products of frame vectors vanish, and $\det(T,N,B_1,B_2)=1$;
\item with $\kappa_3:=\ip{B_1'}{B_2}$ the \emph{Cartan equations}
\begin{equation}\label{eq:cartan}
\begin{aligned}
T'&=N, & N'&=\kappa_2T-B_1,\\
B_1'&=-\kappa_2N+\kappa_3B_2, & B_2'&=-\kappa_3T
\end{aligned}
\end{equation}
hold;
\item $\kappa_3=\det(\alpha',\alpha'',\alpha''',\alpha^{(4)})$;
\item if $\{T,N,\tilde B_1,\tilde B_2\}$ is a positively oriented frame along $\alpha$ with the inner products in~(a) and $N'=\tilde\kappa_2T-\tilde B_1$ for some function $\tilde\kappa_2$, then $\tilde\kappa_2=\kappa_2$, $\tilde B_1=B_1$ and $\tilde B_2=B_2$.
\end{enumerate}
\end{theorem}

\begin{proof}
(a) By Lemma~\ref{lem:pa}, $\ip{T}{N}=0$, $\ip{T}{B_1}=-\ip{\alpha'}{\alpha'''}=1$, $\ip{N}{B_1}=-\ip{\alpha''}{\alpha'''}=0$ and $\ip{B_1}{B_1}=2\kappa_2+\ip{\alpha'''}{\alpha'''}=0$. Hence $T,N,B_1$ span a nondegenerate subspace of index one, $X\mapsto\det(T,N,B_1,X)$ is a nonzero linear form, and $B_2$ is a nonzero vector orthogonal to $T,N,B_1$, hence spacelike. Let $U$ be a unit vector orthogonal to $T,N,B_1$ and let $A$ be the matrix with columns $T,N,B_1,U$. Then $A^{\top}\mathrm{diag}(-1,1,1,1)A$ is the Gram matrix of $(T,N,B_1,U)$, whose determinant is $-1$; thus $\det A=\pm1$. Since $B_2=cU$ with $c=\ip{B_2}{U}=\det A$, we get $\ip{B_2}{B_2}=c^2=1$ and $\det(T,N,B_1,B_2)=c\det A=1$.

(b) Every $X\in\E$ satisfies
\begin{equation}\label{eq:expand}
X=\ip{X}{B_1}T+\ip{X}{N}N+\ip{X}{T}B_1+\ip{X}{B_2}B_2 .
\end{equation}
The first two equations in \eqref{eq:cartan} are the definitions of $N$ and $B_1$. Next, $B_1'=\kappa_2'T+\kappa_2N-\alpha^{(4)}$, so by Lemma~\ref{lem:pa} $\ip{B_1'}{T}=0$, $\ip{B_1'}{N}=\kappa_2+\ip{\alpha'''}{\alpha'''}=-\kappa_2$, and $\ip{B_1'}{B_1}=\frac12\ip{B_1}{B_1}'=0$; by \eqref{eq:expand}, $B_1'=-\kappa_2N+\kappa_3B_2$. Finally $\ip{B_2'}{B_2}=0$, $\ip{B_2'}{T}=-\ip{B_2}{N}=0$, $\ip{B_2'}{N}=-\ip{B_2}{\kappa_2T-B_1}=0$ and $\ip{B_2'}{B_1}=-\ip{B_2}{B_1'}=-\kappa_3$, so $B_2'=-\kappa_3T$.

(c) By \eqref{eq:cartan}, $\alpha'''=\kappa_2T-B_1$ and $\alpha^{(4)}=\kappa_2'T+2\kappa_2N-\kappa_3B_2$, hence
\[
\det(\alpha',\alpha'',\alpha''',\alpha^{(4)})=\det(T,N,-B_1,-\kappa_3B_2)=\kappa_3 .
\]

(d) From $N'=\tilde\kappa_2T-\tilde B_1$ we get $\ip{N'}{N'}=-2\tilde\kappa_2$, so $\tilde\kappa_2=\kappa_2$ and $\tilde B_1=\kappa_2T-N'=B_1$. The vector $\tilde B_2$ is a unit vector orthogonal to $T,N,B_1$, hence $\tilde B_2=\pm B_2$, and the orientation condition gives $\tilde B_2=B_2$.
\end{proof}

We call $\{T,N,B_1,B_2\}$ the \emph{Cartan frame} and $\kappa_2,\kappa_3$ the \emph{Cartan curvatures} of $\alpha$ (the first curvature equals $1$ by the choice of parameter). The Cartan frame goes back to Bonnor \cite{bonnor}; for its construction in an arbitrary parametrization see \cite{coken}. With $F=(T,N,B_1,B_2)^{\top}$, Eqs.~\eqref{eq:cartan} read $F'=\mathcal MF$, and we denote by $\mathcal G$ the Gram matrix of the frame:
\[
\mathcal M=\begin{pmatrix}0&1&0&0\\ \kappa_2&0&-1&0\\ 0&-\kappa_2&0&\kappa_3\\ -\kappa_3&0&0&0\end{pmatrix},\qquad
\mathcal G=\begin{pmatrix}0&0&1&0\\ 0&1&0&0\\ 1&0&0&0\\ 0&0&0&1\end{pmatrix}.
\]

\begin{proposition}\label{prop:fund}
Let $\kappa_2,\kappa_3:I\to\R$ be smooth, $s_0\in I$, $p\in\E$, and let $(T_0,N_0,B_1^0,B_2^0)$ be a basis of $\E$ with the inner products of Theorem~\ref{thm:frame}(a) and $\det(T_0,N_0,B_1^0,B_2^0)=1$. Then there is a unique null curve $\alpha:I\to\E$ parametrized by pseudo-arc with $\alpha(s_0)=p$, Cartan frame $(T_0,N_0,B_1^0,B_2^0)$ at $s_0$ and Cartan curvatures $\kappa_2,\kappa_3$. Two null curves with the same Cartan curvatures differ by an isometry of $\E$ of determinant one.
\end{proposition}

\begin{proof}
Let $F=(T,N,B_1,B_2)^{\top}$ be the solution of $F'=\mathcal MF$ with $F(s_0)=(T_0,N_0,B_1^0,B_2^0)^{\top}$. The Gram matrix $\Gamma$ of $F$ satisfies $\Gamma'=\mathcal M\Gamma+\Gamma\mathcal M^{\top}$; since $\mathcal M\mathcal G+\mathcal G\mathcal M^{\top}=0$, the constant $\mathcal G$ solves the same linear equation with the same initial value, so $\Gamma\equiv\mathcal G$. Moreover $(\det F)'=\operatorname{tr}(\mathcal M)\det F=0$. Put $\alpha(s)=p+\int_{s_0}^sT$. Then $\alpha'=T$ is null and $\alpha''=N$ is a unit spacelike vector, so $s$ is a pseudo-arc parameter, and $N'=\kappa_2T-B_1$; by Theorem~\ref{thm:frame}(d), $F$ is the Cartan frame of $\alpha$ and $\kappa_2,\kappa_3$ are its curvatures. If $\beta$ is another such curve, the isometry mapping $\beta(s_0)$ and its Cartan frame at $s_0$ to $\alpha(s_0)$ and the Cartan frame of $\alpha$ at $s_0$ has determinant one, and by uniqueness of solutions it maps $\beta$ onto $\alpha$.
\end{proof}

\begin{proposition}\label{prop:hyperplane}
$\kappa_3\equiv0$ on $I$ if and only if $\alpha$ lies in a hyperplane of $\E$. In this case $B_2$ is constant and $\alpha$ lies in the timelike hyperplane $\alpha(s_0)+B_2^{\perp}$.
\end{proposition}

\begin{proof}
If $\kappa_3\equiv0$, then $B_2'=0$ and $\ip{\alpha-\alpha(s_0)}{B_2}'=\ip{T}{B_2}=0$; $B_2^{\perp}$ is timelike because $B_2$ is spacelike. Conversely, if $\ip{\alpha-p}{U}=0$ for a constant $U\neq0$, then $\alpha',\dots,\alpha^{(4)}$ lie in $U^{\perp}$, so $\kappa_3=\det(\alpha',\alpha'',\alpha''',\alpha^{(4)})=0$ by Theorem~\ref{thm:frame}(c).
\end{proof}

\begin{remark}\label{rem:other}
Suppose that along $\alpha$ there is a frame $\{T,N,\tilde B_1,\tilde B_2\}$ with $\ip{N}{N}=\ip{\tilde B_1}{\tilde B_1}=\ip{T}{\tilde B_2}=1$, $T$ and $\tilde B_2$ null, all other products zero, $\det(T,N,\tilde B_1,\tilde B_2)=1$, and
\begin{equation}\label{eq:other}
T'=N, \qquad N'=\tilde k_2\tilde B_1-\tilde B_2, \qquad \tilde B_1'=\tilde k_3T-\tilde k_2N, \qquad \tilde B_2'=-\tilde k_3\tilde B_1 .
\end{equation}
A direct computation shows that $B_1=\tilde B_2-\tilde k_2\tilde B_1-\frac12\tilde k_2^2T$ and $B_2=-(\tilde B_1+\tilde k_2T)$ satisfy the hypotheses of Theorem~\ref{thm:frame}(d), hence
\[
\kappa_2=-\tfrac12\tilde k_2^2,\qquad \kappa_3=\tilde k_3+\tilde k_2' .
\]
Consequently: (i) a frame of type \eqref{eq:other} can exist only if $\kappa_2\le0$, i.e. $\ip{\alpha'''}{\alpha'''}\ge0$; for the null helix of Example~\ref{ex:53} we have $\kappa_2=3/2$, so it has no such frame; (ii) such a frame is never unique: $\{T,N,\tilde B_1+2\tilde k_2T,\ \tilde B_2-2\tilde k_2\tilde B_1-2\tilde k_2^2T\}$ is another frame of the same type with the same orientation, and its curvatures are $-\tilde k_2$ and $\tilde k_3+2\tilde k_2'$. Thus $\tilde k_2$ is defined only up to sign, whereas $\kappa_2,\kappa_3$ are invariants of $\alpha$. This is why we work with the Cartan frame of Theorem~\ref{thm:frame}, which is uniquely determined by $\alpha$ by part~(d) of that theorem (compare \cite{coken}).
\end{remark}

\subsection{Constant vectors and C-constant normal fields}

\begin{lemma}\label{lem:const}
Write a vector field along $\alpha$ as $W=aT+bN+cB_1+dB_2$, so that $a=\ip{W}{B_1}$, $b=\ip{W}{N}$, $c=\ip{W}{T}$, $d=\ip{W}{B_2}$. Then
\[
W'=(a'+\kappa_2b-\kappa_3d)T+(b'+a-\kappa_2c)N+(c'-b)B_1+(d'+\kappa_3c)B_2 ;
\]
in particular $W$ is constant if and only if
\begin{equation}\label{eq:constsys}
a'=-\kappa_2b+\kappa_3d,\qquad b'=-a+\kappa_2c,\qquad c'=b,\qquad d'=-\kappa_3c .
\end{equation}
\end{lemma}

\begin{proof}
The component formulas follow from \eqref{eq:expand}, and the expression for $W'$ from \eqref{eq:cartan}.
\end{proof}

\begin{definition}\label{def:cconst}
A vector field $V$ along $\alpha$ is \emph{C-constant} if its components with respect to the Cartan frame are constant, and \emph{normal} if $\ip{V}{T}=0$. Since $\ip{V}{T}$ is the $B_1$-component of $V$, a C-constant normal field has the form
\[
V=\lambda T+\mu N+\nu B_2,\qquad \lambda,\mu,\nu\in\R,\qquad \ip{V}{V}=\mu^2+\nu^2 .
\]
It is a unit field if $\mu^2+\nu^2=1$. We say that $V$ is of type~I if $\nu=0$, of type~II if $\mu=0$, and of type~III if $\lambda=0$.
\end{definition}

\begin{lemma}\label{lem:Vder}
For a C-constant normal field $V=\lambda T+\mu N+\nu B_2$,
\begin{align*}
V'&=(\mu\kappa_2-\nu\kappa_3)T+\lambda N-\mu B_1,\\
V''&=(\lambda\kappa_2+\mu\kappa_2'-\nu\kappa_3')T+(2\mu\kappa_2-\nu\kappa_3)N-\lambda B_1-\mu\kappa_3B_2 .
\end{align*}
\end{lemma}

\begin{proof}
This follows from \eqref{eq:cartan}.
\end{proof}

\begin{definition}\label{def:helix}
Let $V\neq0$ be a C-constant normal field along $\alpha$. We say that $\alpha$ is a \emph{normal helix with axis field $V$} if there are a constant vector $W\neq0$ and $c_0\in\R$ with $\ip{V}{W}=c_0$ on $I$; $W$ is then called an \emph{axis}. We put
\[
\Ax(V)=\{W\in\E:\ \ip{V(s)}{W}=0\ \text{for all } s\in I\},
\]
a linear subspace of $\E$.
\end{definition}

\begin{remark}\label{rem:natural}
(a) By Theorem~\ref{thm:frame}(d), the Cartan frame is determined by $\alpha$, and an isometry of $\E$ of determinant one maps the Cartan frame of $\alpha$ to that of its image. Hence C-constant normal fields, and the property of being a normal helix with axis field $\lambda T+\mu N+\nu B_2$, do not depend on any choice of frame and are invariant under such isometries.

(b) For a null curve the normal space $T^{\perp}=\spn\{T,N,B_2\}$ contains the tangent line. In the terminology of Duggal and Bejancu \cite{duggalbej}, $T^{\perp}=\spn\{T\}\oplus S$, where the screen distribution $S=\spn\{N,B_2\}$ is here determined by the Cartan frame. Thus the C-constant normal fields of type~III are exactly the C-constant sections of the screen distribution, and a general C-constant normal field is such a section plus a constant multiple of the tangent. In $\mathbb{E}^3_1$ the corresponding fields $\lambda_0T+N$ of \cite{nes25} form, up to scaling, a one-parameter family; in $\E$ the C-constant normal fields form, up to scaling, a two-parameter family, which contains $T$, $N$ and $B_2$.

(c) If $\beta(s)=\alpha(-s)$, then $\beta$ is again parametrized by pseudo-arc and, by Theorem~\ref{thm:frame}(d), its Cartan frame at $s$ is $\{-T,N,-B_1,B_2\}$ evaluated at $-s$, with curvatures $\kappa_2(-s)$ and $\kappa_3(-s)$. Hence the field $\lambda T+\mu N+\nu B_2$ along $\beta$ is the field $-\lambda T+\mu N+\nu B_2$ along $\alpha$, i.e. reversing the parameter acts by $(\lambda,\mu,\nu)\mapsto(-\lambda,\mu,\nu)$; for a null Cartan helix it interchanges $\Lambda_+$ and $\Lambda_-$ in Theorem~\ref{thm:main} and, up to sign, the vectors $W_+$ and $W_-$ of Proposition~\ref{prop:decomp}, which explains the symmetry between the two cases. An isometry $A$ of determinant $-1$ maps $\alpha$ to a null curve with Cartan frame $\{AT,AN,AB_1,-AB_2\}$ and curvatures $\kappa_2$ and $-\kappa_3$; it therefore acts by $(\nu,\kappa_3)\mapsto(-\nu,-\kappa_3)$, which leaves $\Lambda_+$, $\Lambda_-$ and $\Lambda_0$ unchanged.
\end{remark}

\section{Normal Helices along Null Cartan Curves}\label{sec:main}

Except in Proposition~\ref{prop:nonzero}, we assume throughout this section that $\kappa_3$ has no zeros on $I$; by Proposition~\ref{prop:hyperplane} this excludes in particular the null curves contained in a hyperplane. For a C-constant normal field $V$ we put
\[
\mathcal H(V)=\{W\in\E:\ \ip{V}{W}\ \text{is constant on } I\}.
\]
This is a linear subspace of $\E$ containing $\Ax(V)$; $\alpha$ is a normal helix with axis field $V$ if and only if $\mathcal H(V)\neq0$, and $\Ax(V)$ is the kernel of the linear form $W\mapsto\ip{V}{W}$ on $\mathcal H(V)$.

\subsection{Reduction to linear differential equations}

For a smooth function $f$ on $I$ and a C-constant normal field $V=\lambda T+\mu N+\nu B_2$ we define
\begin{equation}\label{eq:ops}
\begin{aligned}
\delta[f]&=\frac{\kappa_2'f+2\kappa_2f'-f'''}{\kappa_3}, & \mathcal L[f]&=\delta[f]'+\kappa_3f,\\
\mathcal P_V[f]&=\mu f''+\lambda f'-\nu\kappa_3f, & h_V[f]&=\lambda f+\mu f'+\nu\,\delta[f],
\end{aligned}
\end{equation}
and for a constant vector $W$ we write $c_W=\ip{W}{T}$.

\begin{lemma}\label{lem:reduction}
The operators in \eqref{eq:ops} have the following properties.
\begin{enumerate}
\item For every constant vector $W$ the function $c=c_W$ satisfies $\mathcal L[c]=0$ and
\begin{equation}\label{eq:Wc}
W=(\kappa_2c-c'')\,T+c'N+c\,B_1+\delta[c]\,B_2 .
\end{equation}
Conversely, if $\mathcal L[c]=0$, then the right-hand side of \eqref{eq:Wc} is a constant vector $W$ with $c_W=c$. Hence $W\mapsto c_W$ is a linear isomorphism from $\E$ onto $\ker\mathcal L$.
\item For every smooth $f$,
\begin{equation}\label{eq:Tode}
-\kappa_3\mathcal L[f]=f^{(4)}-\frac{\kappa_3'}{\kappa_3}f'''-2\kappa_2f''+\Bigl(\frac{2\kappa_2\kappa_3'}{\kappa_3}-3\kappa_2'\Bigr)f'+\Bigl(\frac{\kappa_2'\kappa_3'}{\kappa_3}-\kappa_2''-\kappa_3^2\Bigr)f .
\end{equation}
In particular, every component $\ip{T}{e}$, $e\in\E$, of the tangent satisfies the fourth-order linear equation $\mathcal L[f]=0$.
\item $\ip{V}{W}=h_V[c_W]$ for every constant vector $W$, and $h_V[f]'=\mathcal P_V[f]+\nu\,\mathcal L[f]$ for every smooth $f$.
\end{enumerate}
\end{lemma}

\begin{proof}
(a) Write $W=aT+bN+cB_1+dB_2$ as in Lemma~\ref{lem:const}, so that $c=c_W$. If $W$ is constant, Eqs.~\eqref{eq:constsys} give $b=c'$, $a=\kappa_2c-b'=\kappa_2c-c''$ and $\kappa_3d=a'+\kappa_2b=\kappa_2'c+2\kappa_2c'-c'''$, that is, $d=\delta[c]$; the remaining equation $d'=-\kappa_3c$ is $\mathcal L[c]=0$. Conversely, if $\mathcal L[c]=0$ and $a,b,d$ are defined by these formulas, then all four equations in \eqref{eq:constsys} hold, so $W$ is constant by Lemma~\ref{lem:const}. The map $W\mapsto c_W$ is linear, and it is injective because \eqref{eq:Wc} gives $W=0$ if $c_W=0$.

(b) This follows by expanding $\delta[f]'$; the last assertion follows from~(a), because $\ip{T}{e}=c_e$.

(c) By Lemma~\ref{lem:const}, $\ip{V}{W}=\lambda c+\mu b+\nu d$, which equals $h_V[c]$ by~(a). Finally, $h_V[f]'=\lambda f'+\mu f''+\nu\,\delta[f]'$ and $\delta[f]'=\mathcal L[f]-\kappa_3f$.
\end{proof}

\begin{theorem}\label{thm:reduction}
Let $V=\lambda T+\mu N+\nu B_2\neq0$ be a C-constant normal field. Then $W\mapsto c_W$ restricts to a linear isomorphism
\[
\mathcal H(V)\longrightarrow\ker\mathcal L\cap\ker\mathcal P_V ,
\]
and $\ip{V}{W}=h_V[c_W]$ for $W\in\mathcal H(V)$. Consequently:
\begin{enumerate}
\item $\alpha$ is a normal helix with axis field $V$ if and only if the equations $\mathcal L[f]=0$ and $\mathcal P_V[f]=0$ have a common nonzero solution;
\item $\dim\mathcal H(V)\le2$ if $\mu\neq0$, $\dim\mathcal H(V)\le1$ if $\mu=0$ and $\lambda\neq0$, and $\mathcal H(V)=0$ if $\lambda=\mu=0$; in particular, $B_2$ is never an axis field;
\item $\dim\mathcal H(V)-1\le\dim\Ax(V)\le\dim\mathcal H(V)$.
\end{enumerate}
\end{theorem}

\begin{proof}
Let $W$ be a constant vector and $c=c_W$. By Lemma~\ref{lem:reduction}, $\mathcal L[c]=0$, $\ip{V}{W}=h_V[c]$ and $h_V[c]'=\mathcal P_V[c]$. Hence $W\in\mathcal H(V)$ if and only if $\mathcal P_V[c]=0$, and the first assertion follows from Lemma~\ref{lem:reduction}(a); (a) is a reformulation. (b) If $\mu\neq0$, $\mathcal P_V$ is a linear operator of order two, so $\dim\ker\mathcal P_V=2$; if $\mu=0\neq\lambda$, it has order one; and if $\lambda=\mu=0$, then $\nu\neq0$ and $\mathcal P_V[f]=-\nu\kappa_3f$, so $\ker\mathcal P_V=0$. (c) holds because $\Ax(V)$ is the kernel of the linear form $W\mapsto\ip{V}{W}$ on $\mathcal H(V)$.
\end{proof}

The standing hypothesis that $\kappa_3$ has no zeros cannot be dropped in Theorem~\ref{thm:reduction}(b). If $\kappa_3\equiv0$, then $B_2$ is constant (Proposition~\ref{prop:hyperplane}), so $\ip{B_2}{W}$ is constant for every $W$; thus $\mathcal H(B_2)=\E$ is four-dimensional, $\Ax(B_2)=B_2^{\perp}$, and $B_2$ is an axis field. More generally, in this case $B_2\in\mathcal H(V)$ for every C-constant normal field $V$, since $\ip{V}{B_2}=\nu$.

The next theorem shows how to construct the normal helices with a given axis field $V$ and a prescribed curvature $\kappa_3$: one chooses a solution $c$ of the equation $\mathcal P_V[c]=0$, of order at most two, which involves only $\kappa_3$, and then $\kappa_2$ is found by two quadratures; away from the zeros of $c_W$, every such curve arises in this way.

\begin{theorem}\label{thm:construct}
Let $V=\lambda T+\mu N+\nu B_2$ be a C-constant normal field with $(\lambda,\mu)\neq(0,0)$, let $\kappa_3$ have no zeros on $I$, let $s_0\in I$, and let $c$ be a solution of $\mu c''+\lambda c'-\nu\kappa_3c=0$ without zeros on $I$. For $C_1,C_2\in\R$ put
\begin{equation}\label{eq:construct}
D=C_1-\int_{s_0}^s\kappa_3c,\qquad \kappa_2=\frac{1}{c^2}\Bigl(C_2+\int_{s_0}^sc\,(\kappa_3D+c''')\Bigr).
\end{equation}
\begin{enumerate}
\item Every null Cartan curve with curvatures $\kappa_2,\kappa_3$ is a normal helix with axis field $V$: the vector $W=(\kappa_2c-c'')T+c'N+cB_1+DB_2$ is constant, $c_W=c$, and $\ip{V}{W}=\lambda c+\mu c'+\nu D$.
\item Conversely, if $\alpha$ is a normal helix with axis field $V$ and $W\in\mathcal H(V)$ is such that $c_W$ has no zeros, then $\kappa_2$ is given by \eqref{eq:construct} with $c=c_W$, $C_1=\ip{W}{B_2(s_0)}$ and $C_2=\kappa_2(s_0)c_W(s_0)^2$.
\end{enumerate}
In particular, for every C-constant normal field $V$ with $(\lambda,\mu)\neq(0,0)$ and every function $\kappa_3$ without zeros, every point of $I$ has a neighborhood on which there are null Cartan curves with third curvature $\kappa_3$ that are normal helices with axis field $V$; if $\mu=0$, the whole interval $I$ can be taken.
\end{theorem}

\begin{proof}
(a) By \eqref{eq:construct}, $D'=-\kappa_3c$ and $(\kappa_2c^2)'=c(\kappa_3D+c''')$, that is, $\kappa_2'c+2\kappa_2c'-c'''=\kappa_3D$. Hence $\delta[c]=D$ and $\mathcal L[c]=D'+\kappa_3c=0$, so by Lemma~\ref{lem:reduction}(a) $W$ is constant and $c_W=c$. Since $\mathcal P_V[c]=0$, Theorem~\ref{thm:reduction} gives $W\in\mathcal H(V)$, and $\ip{V}{W}=h_V[c]=\lambda c+\mu c'+\nu D$.

(b) Put $c=c_W$ and $d=\ip{W}{B_2}$. By Lemma~\ref{lem:const}, $d'=-\kappa_3c$, so $d=D$ with $C_1=d(s_0)$; by Lemma~\ref{lem:reduction}(a), $\kappa_3d=\kappa_2'c+2\kappa_2c'-c'''$, i.e. $(\kappa_2c^2)'=c(\kappa_3D+c''')$, and integration from $s_0$ gives \eqref{eq:construct}.

For the last assertion take the solution $c$ with $c(s_0)=1$ and, if $\mu\neq0$, $c'(s_0)=0$. If $\mu=0$, then $\lambda\neq0$ and $c=\exp\bigl(\frac{\nu}{\lambda}\int_{s_0}^s\kappa_3\bigr)$ has no zeros on $I$; if $\mu\neq0$, $c$ has no zeros near $s_0$.
\end{proof}

\begin{remark}\label{rem:ktype}
For $V=T$, $N$ and $B_2$, the curves for which $\ip{V}{W}$ is constant are the $k$-type null slant helices of \cite{nes15} with $k=0$, $1$ and $3$, and Theorems~\ref{thm:reduction} and~\ref{thm:construct} recover their characterizations \cite[Theorems~1, 4 and~9, Corollary~6]{nes15}. Indeed, $B_2$ is never an axis field by Theorem~\ref{thm:reduction}(b). For $V=T$ we have $\ker\mathcal P_T=\R$ and $\mathcal L[1]=E_1:=(\kappa_2'/\kappa_3)'+\kappa_3$; by \eqref{eq:construct} with $c=1$, $T$ is an axis field if and only if $\kappa_2=m-\frac12K_3^2$ for a constant $m$ and an antiderivative $K_3$ of $\kappa_3$. Then $\mathcal H(T)=\R W_1$ with
\begin{gather*}
W_1=\kappa_2T+B_1-K_3B_2=\kappa_2T+B_1+\frac{\kappa_2'}{\kappa_3}B_2,\\
\ip{T}{W_1}=1,\qquad \ip{N}{W_1}=0 .
\end{gather*}
For $V=N$ we have $\ker\mathcal P_N=\spn\{1,s\}$, so $N$ is an axis field if and only if $\mathcal L[c_0s+K]=0$ for some $(c_0,K)\neq(0,0)$, and then $\ip{N}{W}=c_0$ for the corresponding axis. For $c_0=0$ this is again $E_1=0$, with axis $W_1$; for $c_0\neq0$ we may assume $K=0$ after a translation of the parameter, and $\mathcal L[s]=0$ integrates to $s\kappa_2'+2\kappa_2+\kappa_3\int s\kappa_3=0$. For instance, for $\kappa_3=1$ and $\kappa_2=-s^2/8$ we have $\mathcal L[s]=0$ and $E_1=\frac34$, so $\mathcal H(N)$ is spanned by $W=-\frac{s^3}{8}T+N+sB_1-\frac{s^2}{2}B_2$, $\ip{N}{W}=1$ and $\Ax(N)=0$.
\end{remark}

\subsection{Null Cartan helices}
A null Cartan curve is called a \emph{null Cartan helix} if $\kappa_2$ and $\kappa_3$ are constant and $\kappa_3\neq0$.

\begin{proposition}\label{prop:spectral}
Let $\kappa_2$ and $\kappa_3\neq0$ be constant. The characteristic polynomial of $\mathcal M$ is $q(\theta)=\theta^4-2\kappa_2\theta^2-\kappa_3^2$; its roots are $\pm r$ and $\pm i\omega$, where
\begin{equation}\label{eq:rw}
r=\Bigl(\kappa_2+\sqrt{\kappa_2^2+\kappa_3^2}\Bigr)^{1/2},\qquad \omega=\Bigl(-\kappa_2+\sqrt{\kappa_2^2+\kappa_3^2}\Bigr)^{1/2}.
\end{equation}
Thus $r,\omega>0$, $r^2-\omega^2=2\kappa_2$, $r\omega=|\kappa_3|$ and $r^2+\omega^2=2\sqrt{\kappa_2^2+\kappa_3^2}$; moreover, the solutions of $\mathcal L[f]=0$ are the linear combinations of $e^{rs}$, $e^{-rs}$, $\cos\omega s$ and $\sin\omega s$.
\end{proposition}

\begin{proof}
Expanding $\det(\theta I-\mathcal M)$ along the first row gives $\theta^2(\theta^2-\kappa_2)-\kappa_2\theta^2-\kappa_3^2$. As a polynomial in $\theta^2$ it has the roots $\kappa_2\pm\sqrt{\kappa_2^2+\kappa_3^2}$, one positive and one negative, which gives \eqref{eq:rw} and the stated relations. By \eqref{eq:Tode}, for constant curvatures $-\kappa_3\mathcal L[f]=f^{(4)}-2\kappa_2f''-\kappa_3^2f$, whose characteristic polynomial is $q$.
\end{proof}

\begin{proposition}\label{prop:normalform}
For $r,\omega>0$ and $\epsilon\in\{1,-1\}$ the curve
\begin{equation}\label{eq:model}
\alpha_{r,\omega,\epsilon}(s)=\frac{1}{\sqrt{r^2+\omega^2}}\Bigl(\frac{\sinh rs}{r},\ \frac{\cosh rs}{r},\ \frac{\sin\omega s}{\omega},\ -\epsilon\,\frac{\cos\omega s}{\omega}\Bigr)
\end{equation}
is a null curve parametrized by pseudo-arc with $\kappa_2=\frac12(r^2-\omega^2)$ and $\kappa_3=\epsilon r\omega$. Consequently every null Cartan helix is congruent, by an isometry of $\E$ of determinant one, to $\alpha_{r,\omega,\epsilon}$ with $r,\omega$ given by \eqref{eq:rw} and $\epsilon=\sgn\kappa_3$.
\end{proposition}

\begin{proof}
Here $T=(r^2+\omega^2)^{-1/2}(\cosh rs,\sinh rs,\cos\omega s,\epsilon\sin\omega s)$, so $\ip{T}{T}=0$, $\ip{T'}{T'}=1$ and $\ip{T''}{T''}=(\omega^4-r^4)/(r^2+\omega^2)=\omega^2-r^2$. By \eqref{eq:defframe}, $\kappa_2=\frac12(r^2-\omega^2)$, and Theorem~\ref{thm:frame}(c) gives $\kappa_3=\det(T,T',T'',T''')=\epsilon r\omega$. The last assertion follows from Propositions~\ref{prop:spectral} and~\ref{prop:fund}.
\end{proof}

Null helices in Lorentzian space forms were studied in \cite{fgl01}; Proposition~\ref{prop:normalform} gives the normal form in the present normalization, which we use below. Note that every pair $(\kappa_2,\kappa_3)$ with $\kappa_3\neq0$ occurs; for instance $\kappa_2>0$ corresponds to $r>\omega$ (Example~\ref{ex:53}).

\begin{proposition}\label{prop:decomp}
Let $\alpha$ be a null Cartan helix and let $r,\omega$ be given by \eqref{eq:rw}. Define
\begin{equation}\label{eq:WZ}
\begin{aligned}
W_+&=e^{-rs}\Bigl((\kappa_2-r^2)T-rN+B_1+\frac{\kappa_3}{r}B_2\Bigr),\\
W_-&=e^{rs}\Bigl((\kappa_2-r^2)T+rN+B_1-\frac{\kappa_3}{r}B_2\Bigr),\\
Z_1&=(\kappa_2+\omega^2)T+B_1,\qquad Z_2=\omega N+\frac{\kappa_3}{\omega}B_2 .
\end{aligned}
\end{equation}
Then $W_+$ and $W_-$ are constant null vectors with $\ip{W_+}{W_-}=-2(r^2+\omega^2)$, $Z_1'=\omega Z_2$, $Z_2'=-\omega Z_1$, $\ip{Z_1}{Z_1}=\ip{Z_2}{Z_2}=r^2+\omega^2$, $\ip{Z_1}{Z_2}=0$, and $Z_1,Z_2$ are orthogonal to $W_\pm$. Consequently $\Sigma=\spn\{Z_1(s),Z_2(s)\}$ is a constant spacelike plane, $\Sigma=\spn\{W_+,W_-\}^{\perp}$, $\E=\spn\{W_+,W_-\}\oplus\Sigma$, and for $s,s_0\in I$
\begin{equation}\label{eq:rot}
\begin{aligned}
Z_1(s_0)&=\cos\omega(s-s_0)\,Z_1(s)-\sin\omega(s-s_0)\,Z_2(s),\\
Z_2(s_0)&=\sin\omega(s-s_0)\,Z_1(s)+\cos\omega(s-s_0)\,Z_2(s).
\end{aligned}
\end{equation}
In particular $c_{W_\pm}=e^{\mp rs}$, $c_{Z_1(s_0)}=\cos\omega(s-s_0)$ and $c_{Z_2(s_0)}=\sin\omega(s-s_0)$.
\end{proposition}

\begin{proof}
Let $Y_\pm$ denote the brackets in the definitions of $W_\pm$, so that $W_\pm=e^{\mp rs}Y_\pm$. By \eqref{eq:cartan},
\[
Y_+'=-\Bigl(r\kappa_2+\frac{\kappa_3^2}{r}\Bigr)T-r^2N+rB_1+\kappa_3B_2,
\]
and $-(r\kappa_2+\kappa_3^2/r)=r(\kappa_2-r^2)$ because $r^4-2\kappa_2r^2-\kappa_3^2=0$. Hence $Y_+'=rY_+$ and $W_+'=0$; similarly $Y_-'=-rY_-$ and $W_-'=0$. Using $r^2-\omega^2=2\kappa_2$ and $r^2\omega^2=\kappa_3^2$ we get $\ip{Y_\pm}{Y_\pm}=2(\kappa_2-r^2)+r^2+\omega^2=0$ and $\ip{Y_+}{Y_-}=2(\kappa_2-r^2)-r^2-\omega^2=-2(r^2+\omega^2)$. Next, $Z_1'=\omega^2N+\kappa_3B_2=\omega Z_2$ and $Z_2'=(\omega\kappa_2-\kappa_3^2/\omega)T-\omega B_1=-\omega Z_1$, since $\kappa_3^2/\omega^2-\kappa_2=r^2-\kappa_2=\kappa_2+\omega^2$. Moreover $\ip{Z_1}{Z_1}=2(\kappa_2+\omega^2)=r^2+\omega^2$, $\ip{Z_2}{Z_2}=\omega^2+\kappa_3^2/\omega^2=\omega^2+r^2$, $\ip{Z_1}{Z_2}=0$, $\ip{Z_1}{Y_\pm}=2\kappa_2+\omega^2-r^2=0$ and $\ip{Z_2}{Y_\pm}=\mp r\omega\pm\kappa_3^2/(r\omega)=0$. Since $W_\pm$ are null with $\ip{W_+}{W_-}\neq0$, $\spn\{W_+,W_-\}$ is a timelike plane; its orthogonal complement is a constant spacelike plane which contains the orthogonal nonzero vectors $Z_1(s),Z_2(s)$, hence equals $\Sigma$. Next, \eqref{eq:rot} follows from $Z_1'=\omega Z_2$, $Z_2'=-\omega Z_1$. Finally, $c_X$ is the $B_1$-coefficient of $X$; it equals $e^{\mp rs}$ for $W_\pm$, $1$ for $Z_1(s)$ and $0$ for $Z_2(s)$, and \eqref{eq:rot} gives the last two formulas.
\end{proof}

By Proposition~\ref{prop:normalform}, $\alpha$ is an orbit of a one-parameter group of isometries which acts as a boost of rapidity $rs$ on the timelike plane $\spn\{W_+,W_-\}$ and as a rotation by the angle $\omega s$ on $\Sigma$; the null lines $\R W_\pm$ and the plane $\Sigma$ are the invariant subspaces of this group. This structure is specific to dimension four: in $\mathbb{E}^3_1$ the Cartan matrix of a null Cartan helix has the eigenvalue $0$ and a single pair $\pm\theta_0$ with $\theta_0$ real or purely imaginary, so a boost and a rotation never occur together.

\begin{theorem}\label{thm:main}
Let $\alpha$ be a null Cartan helix, let $r,\omega$ be given by \eqref{eq:rw}, and let $V=\lambda T+\mu N+\nu B_2\neq0$ be a C-constant normal field. Put
\begin{equation}\label{eq:Lambda}
\Lambda_+=\mu r^2+\lambda r-\nu\kappa_3,\qquad \Lambda_-=\mu r^2-\lambda r-\nu\kappa_3,\qquad \Lambda_0=\mu\omega^2+\nu\kappa_3 .
\end{equation}
\begin{enumerate}
\item If $W$ is a constant vector and $\ip{V}{W}$ is constant, then $\ip{V}{W}=0$; that is, $\mathcal H(V)=\Ax(V)$.
\item The axis space is
\begin{multline*}
\Ax(V)=\{\,w_+W_++w_-W_-+z:\ w_+\Lambda_-=0,\ w_-\Lambda_+=0,\\
z\in\Sigma,\ z=0 \text{ unless } \lambda=\Lambda_0=0\,\}.
\end{multline*}
\end{enumerate}
In particular, $\alpha$ is a normal helix with axis field $V$ if and only if
\[
\Lambda_+=0\qquad\text{or}\qquad \Lambda_-=0\qquad\text{or}\qquad \lambda=\Lambda_0=0 ;
\]
in these cases $W_-\in\Ax(V)$, $W_+\in\Ax(V)$ and $\Sigma\subset\Ax(V)$, respectively, and $\dim\Ax(V)\le2$.
\end{theorem}

\begin{proof}
For constant curvatures, $\mathcal P_V$ has constant coefficients and the characteristic polynomial $p(\theta)=\mu\theta^2+\lambda\theta-\nu\kappa_3$, so $\mathcal P_V[e^{\theta s}]=p(\theta)e^{\theta s}$; moreover
\[
p(r)=\Lambda_+,\qquad p(-r)=\Lambda_-,\qquad p(\pm i\omega)=-\Lambda_0\pm i\lambda\omega .
\]
Let $W$ be constant. By Propositions~\ref{prop:spectral} and~\ref{prop:decomp}, $W=w_+W_++w_-W_-+z$ with $z\in\Sigma$ and
\begin{gather*}
c_W=w_+e^{-rs}+w_-e^{rs}+c_z,\\
c_z\in\spn\{\cos\omega s,\sin\omega s\}=\spn\{e^{i\omega s},e^{-i\omega s}\},
\end{gather*}
where $c_z=0$ only if $z=0$. Applying $\mathcal P_V$ termwise and using the linear independence of $e^{rs}$, $e^{-rs}$, $e^{i\omega s}$, $e^{-i\omega s}$, we see that $\mathcal P_V[c_W]=0$ if and only if $w_+\Lambda_-=0$, $w_-\Lambda_+=0$, and $z=0$ unless $p(i\omega)=p(-i\omega)=0$, i.e. unless $\lambda=\Lambda_0=0$. By Theorem~\ref{thm:reduction}, this describes $\mathcal H(V)$. Moreover $h_V$ has constant coefficients, so $h_V[e^{\theta s}]$ is a constant multiple of $e^{\theta s}$; hence for $W\in\mathcal H(V)$ the constant $\ip{V}{W}=h_V[c_W]$ is a linear combination of $e^{\pm rs}$ and $e^{\pm i\omega s}$, and since these functions and $1$ are linearly independent, $\ip{V}{W}=0$. This proves (a) and (b). If $\lambda=\Lambda_0=0$, then $\nu\kappa_3=-\mu\omega^2$ and $\Lambda_+=\Lambda_-=\mu(r^2+\omega^2)$, which is nonzero since $\mu=0$ would give $\nu=0$ and $V=0$; hence in this case $\Ax(V)=\Sigma$. Otherwise $\Ax(V)\subset\spn\{W_+,W_-\}$. In both cases $\dim\Ax(V)\le2$.
\end{proof}

Theorem~\ref{thm:main} has a transparent meaning in terms of this symmetry. Since $\kappa_2$ and $\kappa_3$ are constant, there is a constant skew-adjoint endomorphism $\Omega$ of $\E$ with $T'=\Omega T$, $N'=\Omega N$, $B_1'=\Omega B_1$ and $B_2'=\Omega B_2$; it acts on the Cartan frame through the matrix $\mathcal M$, generates the one-parameter group above, and satisfies $V(s)=\exp(s\Omega)V(0)$ for every C-constant field $V$. In the language of relativity, $\Omega$ is a non-null bivector, of the same algebraic type as the field tensor of a generic constant electromagnetic field: its eigenvalues are $\pm r$ and $\pm i\omega$, its two invariants are $\operatorname{tr}\Omega^2=2(r^2-\omega^2)=4\kappa_2$ and $\det\Omega=-r^2\omega^2=-\kappa_3^2$, the null lines $\R W_\pm$ are its principal null directions, with $\Omega W_\pm=\pm rW_\pm$, and $\Sigma$ is its plane of rotation. As $\ip{V(s)}{W}=\ip{V(0)}{\exp(-s\Omega)W}$, a nonzero constant could only come from a component of $W$ in the kernel of $\Omega$, which is trivial because $\kappa_3\neq0$; this is Theorem~\ref{thm:main}(a). The three cases of the theorem say that $V$ is orthogonal to one of the two principal null directions ($\Lambda_+=0$ or $\Lambda_-=0$), or that $V(s)$ stays in the boost plane $\spn\{W_+,W_-\}=\Sigma^{\perp}$ ($\lambda=\Lambda_0=0$).

For $\kappa_3\equiv0$ and constant $\kappa_2$, on the other hand, the kernel of $\Omega$ is the plane spanned by the constant vectors $\kappa_2T+B_1$ and $B_2$, and every C-constant normal field satisfies $\ip{V}{\kappa_2T+B_1}=\lambda$ and $\ip{V}{B_2}=\nu$. For $\kappa_2<0$ these curves are exactly the classical trajectories of positive mass in the particle model of \cite{nersessian}: in our notation its equations of motion give $\kappa_3\equiv0$ and $\kappa_2=-M^2/(2a_0^2)$, where $M$ is the mass and $a_0$ the coupling constant in front of the action, the conserved momentum is $p=-a_0(\kappa_2T+B_1)$, with $\ip{p}{p}=2a_0^2\kappa_2=-M^2$, and the Pauli--Lubanski vector is $\pm a_0^2B_2$. Along such a trajectory $\ip{V}{p}=-a_0\lambda$, so the $T$-component of $V$ is, up to the factor $-a_0$, its projection on the conserved momentum, and $\nu$ is its projection on the spin direction $B_2$. For a null Cartan helix no conserved direction exists, and the axes of $V$ are instead the invariant subspaces of $\Omega$.

\begin{corollary}\label{cor:unit}
Let $\alpha$ be a null Cartan helix. A unit C-constant normal field $V=\lambda T+\mu N+\nu B_2$ is an axis field of $\alpha$ if and only if one of the following holds:
\begin{enumerate}
\renewcommand{\theenumi}{\roman{enumi}}
\item $\lambda=(\nu\kappa_3-\mu r^2)/r$; then $W_-\in\Ax(V)$;
\item $\lambda=(\mu r^2-\nu\kappa_3)/r$; then $W_+\in\Ax(V)$;
\item $V=\pm V_2$,
\end{enumerate}
where
\[
V_1=\frac{\kappa_3N+r^2B_2}{\sqrt{\kappa_3^2+r^4}},\qquad V_2=\frac{\kappa_3N-\omega^2B_2}{\sqrt{\kappa_3^2+\omega^4}} .
\]
Thus the unit axis fields form two closed curves $\mathcal C_+$, $\mathcal C_-$, graphs of $\lambda$ over the unit circle $\mu^2+\nu^2=1$, which meet exactly at $\pm V_1$, together with the pair $\pm V_2$. Moreover $\ip{V_1}{V_2}=0$, $V_1''=-\omega^2V_1$, $V_2''=r^2V_2$, $\Ax(V_1)=\spn\{W_+,W_-\}$, $\Ax(V_2)=\Sigma$, and $\Ax(V)$ is a line for $V\in(\mathcal C_+\cup\mathcal C_-)\setminus\{\pm V_1\}$.
\end{corollary}

\begin{proof}
The three conditions are $\Lambda_+=0$, $\Lambda_-=0$ and $\lambda=\Lambda_0=0$. Both $\Lambda_\pm$ vanish if and only if $\lambda=0$ and $\mu r^2=\nu\kappa_3$, i.e. $V=\pm V_1$; and $\lambda=\Lambda_0=0$ if and only if $\mu\omega^2=-\nu\kappa_3$, i.e. $V=\pm V_2$. For $V_1$ we have $\Lambda_0\neq0$, and for $V_2$ we have $\Lambda_\pm=\kappa_3(r^2+\omega^2)/\sqrt{\kappa_3^2+\omega^4}\neq0$; the statements on $\Ax(V_1)$, $\Ax(V_2)$ and on the remaining points follow from Theorem~\ref{thm:main}(b). Next, $\ip{V_1}{V_2}$ is a positive multiple of $\kappa_3^2-r^2\omega^2=0$. For $\lambda=0$ Lemma~\ref{lem:Vder} gives $V''=(2\mu\kappa_2-\nu\kappa_3)N-\mu\kappa_3B_2$; with $(\mu,\nu)\propto(\kappa_3,r^2)$ and $(\mu,\nu)\propto(\kappa_3,-\omega^2)$ and $2\kappa_2=r^2-\omega^2$, $\kappa_3^2=r^2\omega^2$ this yields $V_1''=-\omega^2V_1$ and $V_2''=r^2V_2$.
\end{proof}

\begin{corollary}\label{cor:types}
Let $\alpha$ be a null Cartan helix.
\begin{enumerate}
\item (Type I) $V=\lambda T+N$ is an axis field if and only if $\lambda=r$ (axis $W_+$) or $\lambda=-r$ (axis $W_-$); equivalently $\lambda^4-2\kappa_2\lambda^2-\kappa_3^2=0$.
\item (Type II) $V=\lambda T+B_2$ is an axis field if and only if $\lambda=\kappa_3/r$ (axis $W_-$) or $\lambda=-\kappa_3/r$ (axis $W_+$); note that $|\kappa_3|/r=\omega$.
\item (Type III) $V=\mu N+\nu B_2\neq0$ is an axis field if and only if $V$ is a multiple of $V_1$ or of $V_2$; equivalently $\nu\neq0$ and $\rho=\mu/\nu$ satisfies $\kappa_3\rho^2+2\kappa_2\rho-\kappa_3=0$.
\item None of $T$, $N$, $B_2$ is an axis field; that is, there is no constant $W\neq0$ for which $\ip{T}{W}$, $\ip{N}{W}$ or $\ip{B_2}{W}$ is constant.
\end{enumerate}
\end{corollary}

\begin{proof}
Apply Theorem~\ref{thm:main}. (a) Here $\Lambda_\pm=r^2\pm\lambda r$ and $\Lambda_0=\omega^2\neq0$. (b) Here $\Lambda_\pm=\pm\lambda r-\kappa_3$ and $\Lambda_0=\kappa_3\neq0$. (c) Here $\Lambda_+=\Lambda_-$, and the conditions reduce to $(\mu,\nu)\propto(\kappa_3,r^2)$ or $(\mu,\nu)\propto(\kappa_3,-\omega^2)$; the two roots of the quadratic are $\kappa_3/r^2$ and $-\kappa_3/\omega^2$, because their product is $-1=-\kappa_3^2/(r^2\omega^2)$ and their sum is $-2\kappa_2/\kappa_3=(\omega^2-r^2)/\kappa_3$. (d) For $V=T,N,B_2$ we get $(\Lambda_+,\Lambda_-,\lambda,\Lambda_0)=(r,-r,1,0)$, $(r^2,r^2,0,\omega^2)$ and $(-\kappa_3,-\kappa_3,0,\kappa_3)$, respectively.
\end{proof}

\subsection{Fields with \texorpdfstring{$\mu\nu\neq0$}{mu nu not zero}}

For fields with $\mu\nu\neq0$, which include all fields of type~III except the multiples of $N$ and $B_2$, the problem reduces to the two-dimensional solution space of $\mathcal P_V[f]=0$.

\begin{theorem}\label{thm:munu}
Let $V=\lambda T+\mu N+\nu B_2$ be a unit C-constant normal field with $\mu\nu\neq0$, and put
\begin{equation}\label{eq:g01}
g_0=\frac{\lambda}{\mu^2}+\frac{\nu}{\kappa_3}\Bigl(\kappa_2'-\frac{\nu}{\mu}\kappa_3'\Bigr),\qquad g_1=\frac{\mu^2-\nu^2}{\mu}+\frac{2\nu}{\kappa_3}\Bigl(\kappa_2-\frac{\lambda^2}{2\mu^2}\Bigr).
\end{equation}
\begin{enumerate}
\item $h_V[f]=g_0f+g_1f'$ for every solution $f$ of $\mathcal P_V[f]=0$. Hence $W\mapsto c_W$ maps $\mathcal H(V)$ onto the space of solutions $f$ of $\mu f''+\lambda f'-\nu\kappa_3f=0$ for which $g_0f+g_1f'$ is constant, and $\ip{V}{W}=g_0c_W+g_1c_W'$.
\item $\dim\Ax(V)=2$ if and only if $g_0=g_1=0$ on $I$.
\item If $g_1$ has no zeros on $I$ and $u=-g_0/g_1$, then $\Ax(V)\neq0$ if and only if
\begin{equation}\label{eq:riccati}
\mu(u'+u^2)+\lambda u-\nu\kappa_3=0\qquad\text{on } I ;
\end{equation}
in this case $\Ax(V)$ is spanned by the constant vector $W$ with $c_W=\exp\int_{s_0}^su$.
\end{enumerate}
\end{theorem}

\begin{proof}
(a) For a solution $f$ of $\mathcal P_V[f]=0$ we have $f''=(\nu\kappa_3f-\lambda f')/\mu$ and $f'''=(\nu\kappa_3'f+\nu\kappa_3f'-\lambda f'')/\mu$; substituting these into $h_V[f]$ and using $\mu^2+\nu^2=1$ gives $g_0f+g_1f'$. For such $f$, Lemma~\ref{lem:reduction}(c) gives $(g_0f+g_1f')'=h_V[f]'=\nu\,\mathcal L[f]$; since $\nu\neq0$, $f\in\ker\mathcal L$ if and only if $g_0f+g_1f'$ is constant, and the rest follows from Theorem~\ref{thm:reduction}.

(b) If $g_0=g_1=0$, then $h_V$ vanishes on $\ker\mathcal P_V$, hence $\ker\mathcal P_V\subset\ker\mathcal L$ by~(a), and $\Ax(V)=\mathcal H(V)$ is isomorphic to the two-dimensional space $\ker\mathcal P_V$. Conversely, if $\dim\Ax(V)=2$, then by Theorem~\ref{thm:reduction} the functions $c_W$, $W\in\Ax(V)$, form the whole space $\ker\mathcal P_V$, so $g_0f+g_1f'=0$ for all its elements; since $(f(s),f'(s))$ takes all values in $\R^2$ for every $s$, $g_0=g_1=0$.

(c) Let $W\in\Ax(V)$, $W\neq0$, and $f=c_W$. Then $g_0f+g_1f'=0$, and $f$ has no zeros: at a zero of $f$ also $f'=0$, because $g_1\neq0$, and then $f\equiv0$ since $\mathcal P_V[f]=0$ and $\mu\neq0$. Hence $f'/f=u$, and $\mathcal P_V[f]=f\bigl(\mu(u'+u^2)+\lambda u-\nu\kappa_3\bigr)=0$ gives \eqref{eq:riccati}; moreover $f$ is a multiple of $\exp\int_{s_0}^su$, so $\dim\Ax(V)=1$. Conversely, if \eqref{eq:riccati} holds, then $f=\exp\int_{s_0}^su$ satisfies $\mathcal P_V[f]=0$ and $g_0f+g_1f'=f(g_0+g_1u)=0$, so $f=c_W$ for some $W\in\Ax(V)$ by~(a).
\end{proof}

By Corollary~\ref{cor:unit}, the fields $V_1$ and $V_2$ of a null Cartan helix are distinguished by the fact that they move in a fixed plane: $V_1(s)\in\Sigma$ and $V_2(s)\in\spn\{W_+,W_-\}$. The next theorem determines, for arbitrary curvatures, all null Cartan curves carrying a unit C-constant normal field with this property.

\begin{theorem}\label{thm:planar}
Let $V=\lambda T+\mu N+\nu B_2$ be a unit C-constant normal field. Then $V(s)$ lies in a fixed two-dimensional subspace $\Pi$ for all $s\in I$ if and only if $\mu\nu\neq0$ and, for some (hence every) $s_0\in I$,
\begin{equation}\label{eq:planar}
\kappa_3(s)=\kappa_3(s_0)\,e^{2\lambda(s-s_0)/\mu},\qquad \kappa_2(s)=\frac{\nu^2-\mu^2}{2\mu\nu}\,\kappa_3(s)+\frac{\lambda^2}{2\mu^2}.
\end{equation}
In this case $\Pi=\spn\{V(s),V'(s)\}$ for every $s$,
\[
V''=-\frac{\mu\kappa_3}{\nu}\,V+\frac{\lambda}{\mu}\,V',\qquad \ip{V'}{V'}=\frac{\mu\kappa_3}{\nu},
\]
$\Ax(V)=\Pi^{\perp}$ is two-dimensional, and every constant $W$ for which $\ip{V}{W}$ is constant satisfies $\ip{V}{W}=0$. If $\lambda=0$, then $\alpha$ is a null Cartan helix and $V\in\{\pm V_1,\pm V_2\}$; if $\lambda\neq0$, then $\kappa_3$ is not constant.
\end{theorem}

\begin{proof}
By Lemma~\ref{lem:Vder}, if $\mu\neq0$ then $V'$ has the nonzero $B_1$-component $-\mu$ while $V$ has none; if $\mu=0$ then $V=\lambda T+\nu B_2$ with $\nu=\pm1$, and $V'=-\nu\kappa_3T+\lambda N$ is not a multiple of $V$. Hence $V(s)$ and $V'(s)$ are linearly independent, the linear span of $\{V(s):s\in I\}$ has dimension at least two, and, since $\Ax(V)$ is its orthogonal complement, $V(s)$ lies in a fixed two-dimensional subspace $\Pi$ if and only if $\dim\Ax(V)=2$; then $\Pi=\spn\{V(s),V'(s)\}=\Ax(V)^{\perp}$ for every $s$. By Theorem~\ref{thm:reduction}, $\Ax(V)$ corresponds to the functions $f\in\ker\mathcal L\cap\ker\mathcal P_V$ with $h_V[f]=0$, so $\dim\Ax(V)=2$ requires $\mu\neq0$ and $h_V=0$ on the two-dimensional space $\ker\mathcal P_V$. If $\nu=0$, then $h_V[f]=\lambda f+\mu f'$, and the solution of $\mathcal P_V[f]=0$ with $f(s_0)=0$, $f'(s_0)=1$ has $h_V[f](s_0)=\mu\neq0$; hence $\dim\Ax(V)<2$. If $\mu\nu\neq0$, Theorem~\ref{thm:munu}(b) shows that $\dim\Ax(V)=2$ if and only if $g_0=g_1=0$. Now $g_1=0$ is the second identity in \eqref{eq:planar}, and if it holds, then $\mu^2+\nu^2=1$ gives
\[
g_0=\frac{2\lambda\kappa_3-\mu\kappa_3'}{2\mu^2\kappa_3},
\]
so $g_0=0$ if and only if $\kappa_3'=2\lambda\kappa_3/\mu$, which is the first identity in \eqref{eq:planar}.

Assume now \eqref{eq:planar}. By Theorem~\ref{thm:munu}(b), $\mathcal H(V)=\Ax(V)$, so every constant $\ip{V}{W}$ vanishes. With Lemma~\ref{lem:Vder} and \eqref{eq:planar} one checks directly that the components of $V''-\beta V-\gamma V'$ vanish for $\beta=-\mu\kappa_3/\nu$ and $\gamma=\lambda/\mu$, and that $\ip{V'}{V'}=\lambda^2-2\mu(\mu\kappa_2-\nu\kappa_3)=\mu\kappa_3/\nu$. Finally, if $\lambda=0$, \eqref{eq:planar} shows that $\kappa_2,\kappa_3$ are constant and $2\mu\nu\kappa_2=(\nu^2-\mu^2)\kappa_3$, i.e. $\rho=\mu/\nu$ satisfies $\kappa_3\rho^2+2\kappa_2\rho-\kappa_3=0$, so $V\in\{\pm V_1,\pm V_2\}$ by Corollary~\ref{cor:types}(c); if $\lambda\neq0$, $\kappa_3$ is a non-constant exponential.
\end{proof}

\begin{remark}\label{rem:causal}
Since $\ip{V}{V}=1$ and $\ip{V}{V'}=0$, the plane $\Pi$ of Theorem~\ref{thm:planar} is spacelike if $\mu\kappa_3/\nu>0$ and timelike if $\mu\kappa_3/\nu<0$. Accordingly, $V(s)$ runs along the unit circle, respectively along a branch of the unit hyperbola, of $\Pi$ with speed $\sqrt{|\mu\kappa_3/\nu|}$. For a null Cartan helix, $V_1$ gives $\mu\kappa_3/\nu=\kappa_3^2/r^2=\omega^2$ and $\Pi=\Sigma$, while $V_2$ gives $\mu\kappa_3/\nu=-\kappa_3^2/\omega^2=-r^2$ and $\Pi=\spn\{W_+,W_-\}$. In the same way, $V(s)$ lies in a fixed affine plane if and only if $\dim\mathcal H(V)=2$; by Theorem~\ref{thm:planar} this plane passes through the origin only in the situation of \eqref{eq:planar}, and Corollary~\ref{cor:typeI}(e) and Example~\ref{ex:510} show that affine planes not through the origin occur.
\end{remark}

\subsection{Fields of types I and II}

For $V=\lambda T+N$ with $\lambda\neq0$ we have $\ker\mathcal P_V=\spn\{1,e^{-\lambda s}\}$; the next theorem makes Theorem~\ref{thm:reduction} explicit in this case.

\begin{theorem}\label{thm:typeI}
Let $\lambda\neq0$ and $V=\lambda T+N$. Put
\begin{equation}\label{eq:E12}
P=\kappa_2'-2\lambda\kappa_2+\lambda^3,\qquad E_1=\Bigl(\frac{\kappa_2'}{\kappa_3}\Bigr)'+\kappa_3,\qquad E_2=\Bigl(\frac{P}{\kappa_3}\Bigr)'-\lambda\frac{P}{\kappa_3}+\kappa_3 .
\end{equation}
For $c_0,K\in\R$ let $W_{c_0,K}=aT+bN+cB_1+dB_2$ with
\begin{equation}\label{eq:Wc0K}
\begin{aligned}
c&=\frac{c_0}{\lambda}+Ke^{-\lambda s}, & b&=-\lambda Ke^{-\lambda s},\\
a&=\kappa_2c-\lambda^2Ke^{-\lambda s}, & d&=\frac{1}{\kappa_3}\Bigl(\frac{c_0\kappa_2'}{\lambda}+PKe^{-\lambda s}\Bigr).
\end{aligned}
\end{equation}
Then $\ip{V}{W_{c_0,K}}=c_0$ and
\[
W_{c_0,K}'=\Bigl(\frac{c_0}{\lambda}E_1+Ke^{-\lambda s}E_2\Bigr)B_2 .
\]
Every constant vector $W$ for which $\ip{V}{W}$ is constant equals $W_{c_0,K}$ for some $c_0,K$, and $W_{c_0,K}\neq0$ if and only if $(c_0,K)\neq(0,0)$. Consequently, $\alpha$ is a normal helix with axis field $V$ and $\ip{V}{W}=c_0$ if and only if there is $K\in\R$ with $(c_0,K)\neq(0,0)$ such that
\begin{equation}\label{eq:cond}
\frac{c_0}{\lambda}E_1+Ke^{-\lambda s}E_2=0\qquad\text{on } I .
\end{equation}
\end{theorem}

\begin{proof}
We use Lemma~\ref{lem:const}. By \eqref{eq:Wc0K}, $c'=b$ and $b'+a-\kappa_2c=\lambda^2Ke^{-\lambda s}-\lambda^2Ke^{-\lambda s}=0$. Next,
\[
a'+\kappa_2b-\kappa_3d=\kappa_2'c-2\lambda\kappa_2Ke^{-\lambda s}+\lambda^3Ke^{-\lambda s}-\frac{c_0\kappa_2'}{\lambda}-PKe^{-\lambda s}=0,
\]
and
\begin{equation*}
d'+\kappa_3c=\frac{c_0}{\lambda}\Bigl(\Bigl(\frac{\kappa_2'}{\kappa_3}\Bigr)'+\kappa_3\Bigr)+K\Bigl(\Bigl(\frac{Pe^{-\lambda s}}{\kappa_3}\Bigr)'+\kappa_3e^{-\lambda s}\Bigr)
=\frac{c_0}{\lambda}E_1+Ke^{-\lambda s}E_2 .
\end{equation*}
This gives the formula for $W_{c_0,K}'$, and $\ip{V}{W_{c_0,K}}=\lambda c+b=c_0$. Conversely, let $W=aT+bN+cB_1+dB_2$ be constant with $\ip{V}{W}=\lambda c+b=c_0$. By \eqref{eq:constsys}, $c'=b=c_0-\lambda c$, so $c=c_0/\lambda+Ke^{-\lambda s}$ for some $K\in\R$ and $b=-\lambda Ke^{-\lambda s}$; the equations $b'=-a+\kappa_2c$ and $a'=-\kappa_2b+\kappa_3d$ then determine $a$ and $d$ as in \eqref{eq:Wc0K}. If $(c_0,K)\neq(0,0)$, then $c$ is not identically zero, so $W_{c_0,K}\neq0$; if $(c_0,K)=(0,0)$, all components vanish.
\end{proof}

\begin{corollary}\label{cor:typeI}
In the setting of Theorem~\ref{thm:typeI}:
\begin{enumerate}
\item $\alpha$ is a normal helix with axis field $V$ and $\ip{V}{W}=0$ if and only if $E_2\equiv0$, i.e. $Q=P/\kappa_3$ satisfies $Q'=\lambda Q-\kappa_3$. The axes are then the nonzero multiples of
\[
W=e^{-\lambda s}\bigl((\kappa_2-\lambda^2)T-\lambda N+B_1+QB_2\bigr).
\]
\item If $\kappa_2$ and $\kappa_3$ are constant, \eqref{eq:cond} holds if and only if $c_0=0$ and $\lambda^4-2\kappa_2\lambda^2-\kappa_3^2=0$, i.e. $\lambda=\pm r$, in accordance with Corollary~\ref{cor:types}(a).
\item If $\kappa_2\equiv0$, the condition in~(a) is the Bernoulli equation $\lambda^3\kappa_3'=\kappa_3^3-\lambda^4\kappa_3$, whose nowhere vanishing solutions are
\[
\kappa_3=\pm\lambda^2\bigl(1+Ce^{2\lambda s}\bigr)^{-1/2},\qquad C\in\R,
\]
on any interval where $1+Ce^{2\lambda s}>0$. For $C\neq0$ these curves are normal helices with axis field $V$ that are not null Cartan helices.
\item Nonzero constants occur for non-constant curvatures: for $\lambda=1$, $\kappa_2=0$ and $\kappa_3=-1/(1+e^s)$, condition \eqref{eq:cond} holds with $c_0=K=1$ (Example~\ref{ex:56}).
\item $\dim\mathcal H(V)=2$ if and only if $E_1=E_2=0$. This holds, for instance, for $\kappa_3=qe^{\lambda s/2}$ and $\kappa_2=\frac{\lambda^2}{2}-\frac{2\kappa_3^2}{\lambda^2}$ with a constant $q\neq0$. Then $V(s)$ lies in a fixed affine plane which does not pass through the origin, $\Ax(V)=\R W_{0,1}$, and $T$ and $N$ are axis fields as well.
\end{enumerate}
\end{corollary}

\begin{proof}
(a) For $c_0=0$ we need $K\neq0$, and \eqref{eq:cond} becomes $E_2\equiv0$; the axis is $W_{0,1}$. (b) For constant curvatures $E_1=\kappa_3$ and $E_2=(\kappa_3^2-\lambda P)/\kappa_3=(\kappa_3^2+2\kappa_2\lambda^2-\lambda^4)/\kappa_3$ are constants, so by the linear independence of $1$ and $e^{-\lambda s}$, \eqref{eq:cond} holds if and only if $c_0=0$ and $KE_2=0$ with $K\neq0$. Since the equation $y^2-2\kappa_2y-\kappa_3^2=0$ has exactly one positive root, namely $r^2$, this means $\lambda^2=r^2$. (c) For $\kappa_2=0$ we have $Q=\lambda^3/\kappa_3$, and $Q'=\lambda Q-\kappa_3$ is equivalent to $\lambda^3\kappa_3'=\kappa_3^3-\lambda^4\kappa_3$. The substitution $u=\kappa_3^{-2}$ turns it into $u'=2\lambda u-2/\lambda^3$, with general solution $u=\lambda^{-4}(1+Ce^{2\lambda s})$. For $C\neq0$, $\kappa_3$ is not constant. (d) Here $P=1$, $E_1=\kappa_3=-1/(1+e^s)$ and $E_2=(1/\kappa_3)'-1/\kappa_3+\kappa_3=-e^s+1+e^s-1/(1+e^s)=e^s/(1+e^s)$, so $E_1+e^{-s}E_2=0$. (e) By the theorem, $\dim\mathcal H(V)=2$ if and only if $W_{c_0,K}$ is constant for all $c_0,K$, i.e. $E_1=E_2=0$. For the given functions $\kappa_3'=\frac{\lambda}{2}\kappa_3$, so $\kappa_2'/\kappa_3=-2\kappa_3/\lambda$ and $E_1=-\kappa_3+\kappa_3=0$; moreover $P=2\kappa_3^2/\lambda$, $P/\kappa_3=2\kappa_3/\lambda$ and $E_2=\kappa_3-2\kappa_3+\kappa_3=0$. Since $\ip{V}{W_{c_0,K}}=c_0$, we get $\Ax(V)=\R W_{0,1}$, and $V(s)$ lies in the affine plane $\{x:\ \ip{x}{W_{0,1}}=0,\ \ip{x}{W_{\lambda,0}}=\lambda\}$. Finally, $E_1=0$ means that $T$ is an axis field, with axis $W_1=W_{\lambda,0}$, and $\ip{N}{W_1}=0$ (Remark~\ref{rem:ktype}).
\end{proof}

By Corollary~\ref{cor:types}(a), for a null Cartan helix $\dim\mathcal H(V)\le1$ for every field of type~I. Corollary~\ref{cor:typeI}(e) therefore exhibits non-helical null curves with two linearly independent axes for the same field $\lambda T+N$; this is in contrast with the three-dimensional case, where null Cartan helices are the only null Cartan normal helices with two axes \cite{nes25}.

For fields of type~II the problem can be solved in closed form.

\begin{theorem}\label{thm:typeII}
Let $\lambda\neq0$, $\nu\in\{1,-1\}$, $V=\lambda T+\nu B_2$ and $\xi=\nu/\lambda$, and let $K_3$ be an antiderivative of $\kappa_3$. Then $\alpha$ is a normal helix with axis field $V$ if and only if there are constants $m_1,m_2$ with
\begin{equation}\label{eq:typeII}
\kappa_2=\xi\kappa_3'+\frac{\xi^2}{2}\kappa_3^2-\frac{1}{2\xi^2}+m_1e^{-\xi K_3}+m_2e^{-2\xi K_3}.
\end{equation}
In this case $\mathcal H(V)=\R W$ with
\[
W=e^{\xi K_3}\bigl((\kappa_2-\xi\kappa_3'-\xi^2\kappa_3^2)T+\xi\kappa_3N+B_1\bigr)+\bigl(\xi m_1-\xi^{-1}e^{\xi K_3}\bigr)B_2,
\]
and $\ip{V}{W}=m_1/\lambda$. In particular $\Ax(V)\neq0$ if and only if $m_1=0$. For constant curvatures, \eqref{eq:typeII} holds if and only if $m_1=m_2=0$ and $\lambda^2=\omega^2$, in accordance with Corollary~\ref{cor:types}(b).
\end{theorem}

\begin{proof}
Here $\mathcal P_V[f]=\lambda(f'-\xi\kappa_3f)$, so $\ker\mathcal P_V=\R c$ with $c=e^{\xi K_3}$, and by Theorem~\ref{thm:reduction} $\alpha$ is a normal helix with axis field $V$ if and only if $\mathcal L[c]=0$, in which case $\mathcal H(V)$ is spanned by the vector \eqref{eq:Wc} with this $c$. Since $c$ has no zeros, Theorem~\ref{thm:construct} shows that $\mathcal L[c]=0$ if and only if $\kappa_2$ is given by \eqref{eq:construct} for some $C_1,C_2$. Now $c'=\xi\kappa_3c$, $c''=(\xi\kappa_3'+\xi^2\kappa_3^2)c$, $\int\kappa_3c=c/\xi$ and $\int\kappa_3c^2=c^2/(2\xi)$, and $cc'''=(cc''-\frac12c'^2)'$. Hence $D=\tilde C_1-c/\xi$ and
\[
\kappa_2c^2=\tilde C_2+\frac{\tilde C_1}{\xi}c-\frac{c^2}{2\xi^2}+cc''-\frac12c'^2
\]
with constants $\tilde C_1,\tilde C_2$ which are in bijection with $C_1,C_2$. Dividing by $c^2$ gives \eqref{eq:typeII} with $m_1=\tilde C_1/\xi$ and $m_2=\tilde C_2$. By Theorem~\ref{thm:construct}(a), $W=(\kappa_2c-c'')T+c'N+cB_1+DB_2$ with $D=\xi m_1-c/\xi$, which is the stated vector, and $\ip{V}{W}=\lambda c+\nu D=\lambda c+\nu\xi m_1-\lambda c=m_1/\lambda$, because $\nu\xi=\nu^2/\lambda=1/\lambda$. Finally, for constant curvatures $K_3=\kappa_3s+\mathrm{const}$, and since $1$, $e^{-\xi\kappa_3s}$, $e^{-2\xi\kappa_3s}$ are linearly independent, $\kappa_2$ is constant if and only if $m_1=m_2=0$ and $2\xi^2\kappa_2=\xi^4\kappa_3^2-1$, i.e. $\lambda^4+2\kappa_2\lambda^2-\kappa_3^2=0$; since $y^2+2\kappa_2y-\kappa_3^2=0$ has exactly one positive root, namely $\omega^2$, this means $\lambda^2=\omega^2$.
\end{proof}

\begin{proposition}\label{prop:nonzero}
Let $\alpha$ be a null Cartan curve with constant curvatures. There exist a C-constant normal field $V$ and a constant vector $W$ with $\ip{V}{W}=c_0\neq0$ if and only if $\kappa_3=0$. The constancy of $\kappa_3$ cannot be dropped, even if $\kappa_2$ is constant.
\end{proposition}

\begin{proof}
If $\kappa_3\neq0$, this is Theorem~\ref{thm:main}(a). If $\kappa_3=0$, then $B_2$ is constant (Proposition~\ref{prop:hyperplane}) and $\ip{V}{B_2}=\nu$ for every C-constant normal field $V=\lambda T+\mu N+\nu B_2$; take $\nu\neq0$. The last assertion follows from Corollary~\ref{cor:typeI}(d).
\end{proof}

Table~\ref{tab:summary} summarizes the results of this section.

\begin{table}[H]
\caption{Summary of Section~\ref{sec:main} ($\kappa_3$ without zeros). The quantities $r,\omega$, $W_\pm$, $\Sigma$, $\Lambda_\pm,\Lambda_0$ and $V_1,V_2$ are those of Propositions~\ref{prop:spectral} and~\ref{prop:decomp}, Theorem~\ref{thm:main} and Corollary~\ref{cor:unit}.}\label{tab:summary}
\centering
\small
\begin{tabular}{@{}p{0.19\textwidth}p{0.36\textwidth}p{0.38\textwidth}@{}}
\toprule
Field $V$ & Null Cartan helix & Arbitrary curvatures \\
\midrule
\raggedright General\newline $\lambda T+\mu N+\nu B_2$ & \raggedright Axis field iff $\Lambda_+=0$, $\Lambda_-=0$ or $\lambda=\Lambda_0=0$; always $\ip{V}{W}=0$; $\Ax(V)\subset\spn\{W_+,W_-\}$ or $\Ax(V)=\Sigma$ (Theorem~\ref{thm:main}) & \raggedright $\mathcal H(V)\cong\ker\mathcal L\cap\ker\mathcal P_V$ and $\dim\mathcal H(V)\le2$; helices with given $V$ and $\kappa_3$ by two quadratures (Theorems~\ref{thm:reduction} and~\ref{thm:construct}) \tabularnewline[3pt]
\raggedright Type I\newline $\lambda T+N$ & \raggedright Axis field iff $\lambda=\pm r$, then $\Ax(V)=\R W_\pm$; $\dim\mathcal H(V)\le1$ (Corollary~\ref{cor:types}) & \raggedright Criterion \eqref{eq:cond}; Bernoulli equation for $\kappa_2\equiv0$; $\ip{V}{W}\neq0$ and $\dim\mathcal H(V)=2$ occur, the latter unlike in $\mathbb{E}^3_1$ (Theorem~\ref{thm:typeI}, Corollary~\ref{cor:typeI}) \tabularnewline[3pt]
\raggedright Type II\newline $\lambda T\pm B_2$ & \raggedright Axis field iff $|\lambda|=\omega$, then $\Ax(V)=\R W_+$ or $\R W_-$ (Corollary~\ref{cor:types}) & \raggedright Closed form \eqref{eq:typeII} for $\kappa_2$; $\dim\mathcal H(V)\le1$ and $\ip{V}{W}=m_1/\lambda$ (Theorem~\ref{thm:typeII}) \tabularnewline[3pt]
\raggedright Type III\newline $\mu N+\nu B_2$ & \raggedright Axis field iff $V\parallel V_1$, then $\Ax(V)=\spn\{W_+,W_-\}$, or $V\parallel V_2$, then $\Ax(V)=\Sigma$ (Corollaries~\ref{cor:unit} and~\ref{cor:types}) & \raggedright For $\mu\nu\neq0$: Riccati criterion (Theorem~\ref{thm:munu}); $\dim\Ax(V)=2$ only for helices (Theorem~\ref{thm:planar}); the helix criterion fails in general (Examples~\ref{ex:56} and~\ref{ex:58}) \tabularnewline[3pt]
\raggedright Special\newline $T$, $N$, $B_2$ & \raggedright None is an axis field (Corollary~\ref{cor:types}(d)) & \raggedright $k$-type null slant helices, $k=0,1,3$ (Remark~\ref{rem:ktype}); $B_2$ is never an axis field (Theorem~\ref{thm:reduction}) \tabularnewline
\bottomrule
\end{tabular}
\end{table}

\section{Null Cartan Curves on Timelike Hypersurfaces}\label{sec:hyper}

\subsection{The normalized Darboux frame}
Let $M\subset\E$ be a timelike hypersurface with unit normal field $\eta$, $\ip{\eta}{\eta}=1$, and let $\alpha:I\to M$ be a null Cartan curve parametrized by pseudo-arc. Since $T\in TM$, we have $\ip{\eta}{T}=0$, so along $\alpha$
\begin{equation}\label{eq:eta}
\eta=\lambda T+\mu N+\nu B_2,\qquad \mu^2+\nu^2=1,
\end{equation}
with smooth functions $\lambda,\mu,\nu$; we write $\mu=\sin\phi$, $\nu=\cos\phi$ with a smooth function $\phi$. The second fundamental form $h(X,Y)=\ip{\bar D_XY}{\eta}$, where $\bar D$ is the flat connection of $\E$, gives the normal curvature $h(T,T)=\ip{T'}{\eta}=\ip{N}{\eta}=\mu$. The tangential part of $T'=N$ is $\nabla_TT=N-\mu\eta$, where $\nabla$ is the induced connection.

Frames $\{T,\zeta,e\}$ of $TM$ along $\alpha$ with $\zeta$ null, $\ip{T}{\zeta}=1$, $e$ a unit vector orthogonal to $T$ and $\zeta$ are not unique: since $TM\cap T^{\perp}=\spn\{T,e\}$, any other one is obtained by $e\mapsto\pm(e+aT)$, $\zeta\mapsto\zeta-a\,e-\frac{a^2}{2}T$ with a function $a$. This freedom is removed by the following normalization.

\begin{proposition}\label{prop:darboux}
Suppose that $\nu$ has no zeros on $I$, and put $\sigma=\lambda/\nu$,
\begin{equation}\label{eq:darbouxframe}
e=\nu N-\mu B_2-\frac{\lambda\mu}{\nu}T,\qquad \zeta=B_1-\frac{\lambda}{\nu}B_2-\frac{\lambda^2}{2\nu^2}T .
\end{equation}
Then $\{T,\zeta,e\}$ is a frame of $TM$ along $\alpha$, $\{T,\zeta,e,\eta\}$ satisfies $\ip{T}{T}=\ip{\zeta}{\zeta}=0$, $\ip{T}{\zeta}=\ip{e}{e}=\ip{\eta}{\eta}=1$, all other products vanish, and $N=\nu e+\mu\eta$. It is the unique frame of the above type with $\ip{N}{\zeta}=0$ and $\ip{N}{e}=\nu$. It satisfies
\begin{equation}\label{eq:darboux}
\begin{aligned}
T'&=\nu e+\mu\eta, & \zeta'&=\tau_e e+\tau_n\eta,\\
e'&=-\tau_eT-\nu\zeta+\tau\eta, & \eta'&=-\tau_nT-\mu\zeta-\tau e,
\end{aligned}
\end{equation}
where
\begin{equation}\label{eq:taus}
\begin{aligned}
\tau_e&=-\kappa_2\nu-\kappa_3\mu-\frac{\nu\sigma^2}{2}+\mu\sigma',\\
\tau_n&=\kappa_3\nu-\kappa_2\mu-\frac{\mu\sigma^2}{2}-\nu\sigma',\\
\tau&=-(\sigma+\phi').
\end{aligned}
\end{equation}
Conversely,
\begin{equation}\label{eq:inverse}
\begin{aligned}
\lambda&=-\nu(\tau+\phi'),\\
\kappa_2&=-(\nu\tau_e+\mu\tau_n)-\tfrac12(\tau+\phi')^2,\\
\kappa_3&=-(\tau+\phi')'-(\mu\tau_e-\nu\tau_n).
\end{aligned}
\end{equation}
In particular, the pair $(\alpha,M)$ is described along $\alpha$ either by $(\kappa_2,\kappa_3,\lambda,\phi)$ or by $(\phi,\tau_e,\tau_n,\tau)$. Moreover, $\eta$ is C-constant if and only if $\phi$ and $\tau$ are constant, and in this case
\[
\lambda=-\nu\tau,\qquad \kappa_2=-(\nu\tau_e+\mu\tau_n)-\tfrac12\tau^2,\qquad \kappa_3=\nu\tau_n-\mu\tau_e .
\]
\end{proposition}

\begin{proof}
The inner products follow from Theorem~\ref{thm:frame}(a); for instance $\ip{\zeta}{\zeta}=-\lambda^2/\nu^2+\lambda^2/\nu^2=0$, $\ip{e}{\eta}=\nu\mu-\mu\nu=0$ and $\ip{\zeta}{\eta}=\lambda-\lambda=0$. Hence $T,\zeta,e$ lie in $\eta^{\perp}=TM$ and are linearly independent. Moreover $\nu e+\mu\eta=(\nu^2+\mu^2)N=N$. If $\tilde e=\pm(e+aT)$ and $\tilde\zeta=\zeta-a\,e-\frac{a^2}{2}T$, then $\ip{N}{\tilde\zeta}=-a\nu$ and $\ip{N}{\tilde e}=\pm\nu$, so the two normalizations force $a=0$ and the upper sign. The first equation in \eqref{eq:darboux} is $T'=N$. For the last one, \eqref{eq:cartan} gives
\[
\eta'=(\lambda'+\mu\kappa_2-\nu\kappa_3)T+(\lambda+\mu')N-\mu B_1+\nu'B_2,
\]
and the coefficients of $\eta'$ in the frame $\{T,\zeta,e,\eta\}$ are $\ip{\eta'}{\zeta}$, $\ip{\eta'}{T}$, $\ip{\eta'}{e}$, $\ip{\eta'}{\eta}$; for example $\ip{\eta'}{T}=-\mu$ and $\ip{\eta'}{e}=\nu(\lambda+\mu')-\mu\nu'+\lambda\mu^2/\nu=\sigma+\phi'=-\tau$. The equations for $\zeta'$ and $e'$ are obtained in the same way; the skew-symmetry of the coefficient matrix with respect to the Gram matrix of the frame (which is constant) explains the pattern in \eqref{eq:darboux}. Finally, \eqref{eq:taus} gives $\nu\tau_e+\mu\tau_n=-\kappa_2-\sigma^2/2$ and $\mu\tau_e-\nu\tau_n=\sigma'-\kappa_3$, and $\sigma=-(\tau+\phi')$; this yields \eqref{eq:inverse}. Finally, $\mu=\sin\phi$ and $\nu=\cos\phi$ are constant if and only if $\phi$ is; in this case $\lambda=-\nu\tau$ by \eqref{eq:inverse}, so $\lambda$ is constant if and only if $\tau$ is, and \eqref{eq:inverse} gives the formulas for $\kappa_2$ and $\kappa_3$.
\end{proof}

Thus $\kappa_n:=\ip{T'}{\eta}=\mu=\sin\phi$ and $\kappa_e:=\ip{T'}{e}=\nu=\cos\phi$ play the roles of normal and geodesic curvature, and $\tau_e$, $\tau_n$, $\tau$ are torsion-type invariants; in particular, $\tau=\ip{e'}{\eta}$ is the analogue of the geodesic torsion.

\begin{corollary}\label{cor:asym}
Let $\alpha$ be a null Cartan curve on a timelike hypersurface $M$ with unit normal \eqref{eq:eta} along $\alpha$.
\begin{enumerate}
\item $h(T,T)=\mu$; hence $\alpha$ is an asymptotic curve of $M$ if and only if $\mu\equiv0$, i.e. $\eta=\lambda T\pm B_2$.
\item $\alpha$ is a pregeodesic of $M$ (i.e. $\nabla_TT$ is parallel to $T$) if and only if $\nu\equiv0$, i.e. $\eta=\lambda T\pm N$.
\end{enumerate}
\end{corollary}

\begin{proof}
(a) was shown above. (b) $\nabla_TT=N-\mu\eta=\nu(\nu N-\mu B_2)-\lambda\mu T$, which is parallel to $T$ if and only if $\nu=0$.
\end{proof}

In case~(b) the frame of Proposition~\ref{prop:darboux} is not defined, because then $N\in\spn\{T,\eta\}$ and the normalization $N=\nu e+\mu\eta$ cannot be achieved.

\begin{proposition}\label{prop:ruled}
Let $\alpha$ be a null Cartan curve and let $\lambda,\mu,\nu$ be smooth functions on $I$ with $\mu^2+\nu^2=1$. Put $\eta=\lambda T+\mu N+\nu B_2$,
\begin{gather*}
X=\nu N-\mu B_2,\qquad Y=B_1-\lambda(\mu N+\nu B_2),\\
\Phi(s,u,v)=\alpha(s)+uX(s)+vY(s).
\end{gather*}
For every compact subinterval $J\subset I$ there is $\varepsilon>0$ such that $\Phi$ is an immersion of $J\times(-\varepsilon,\varepsilon)^2$ onto a timelike hypersurface $M$ containing $\alpha(J)$ whose unit normal along $\alpha$ is $\eta$.
\end{proposition}

\begin{proof}
At $(s,0,0)$ we have $\Phi_s=T$, $\Phi_u=X$, $\Phi_v=Y$. These vectors are orthogonal to $\eta$, and their Gram matrix is
\[
\begin{pmatrix}0&0&1\\0&1&0\\1&0&\lambda^2\end{pmatrix},
\]
whose determinant is $-1$. Hence they are linearly independent and span the Lorentzian subspace $\eta^{\perp}$. Since being an immersion and having a Lorentzian induced metric are open conditions, the assertion follows from the compactness of $J$.
\end{proof}

\subsection{Isophotic and silhouette curves}

\begin{definition}\label{def:iso}
Let $W\neq0$ be a constant vector. A curve $\alpha\subset M$ is an \emph{isophotic curve} of $M$ with respect to $W$ if $\ip{\eta}{W}$ is constant along $\alpha$, and a \emph{silhouette curve} if this constant is zero.
\end{definition}

For curves on surfaces in $\mathbb{E}^3_1$, isophote curves have been studied by means of Darboux frames in \cite{doganyayli}, and null Cartan normal isophotic and normal silhouette curves in \cite{djnes24}. If the unit normal \eqref{eq:eta} is C-constant along $\alpha$, i.e. $\lambda,\mu,\nu$ are constant (by Proposition~\ref{prop:darboux}, if $\nu$ has no zeros this means that the invariants $\phi$ and $\tau$ are constant), then $\alpha$ is an isophotic curve with respect to $W$ if and only if $W\in\mathcal H(\eta)$, and a silhouette curve if and only if $W\in\Ax(\eta)$. Section~\ref{sec:main} therefore yields:

\begin{theorem}\label{thm:iso}
Let $\alpha$ be a null Cartan curve with $\kappa_3$ nowhere zero on a timelike hypersurface $M$ whose unit normal along $\alpha$ is C-constant, $\eta=\lambda T+\mu N+\nu B_2$, and let $W\neq0$ be a constant vector.
\begin{enumerate}
\item The vectors $W$ with respect to which $\alpha$ is an isophotic curve of $M$ form, together with $0$, the subspace $\mathcal H(\eta)$, of dimension at most two, which is described by Theorems~\ref{thm:reduction}, \ref{thm:munu}, \ref{thm:typeI} and~\ref{thm:typeII}.
\item If $\alpha$ is a null Cartan helix, then $\alpha$ is an isophotic curve of $M$ with respect to $W$ if and only if it is a silhouette curve with respect to $W$, if and only if $W\in\Ax(\eta)$, the space described in Theorem~\ref{thm:main}(b). Such a $W$ exists if and only if $\Lambda_+=0$, $\Lambda_-=0$ or $\lambda=\Lambda_0=0$, that is, $\eta\in\mathcal C_+\cup\mathcal C_-\cup\{\pm V_2\}$ in the notation of Corollary~\ref{cor:unit}.
\item If $\alpha$ is an asymptotic curve, i.e. $\eta=\lambda T+\nu B_2$ with $\nu=\pm1$, then $\alpha$ is an isophotic curve of $M$ with respect to some $W$ if and only if $\lambda\neq0$ and $\kappa_2$ has the form \eqref{eq:typeII}; such a $W$ is unique up to a factor, and $\alpha$ is a silhouette curve with respect to it if and only if $m_1=0$. In terms of the invariants of Proposition~\ref{prop:darboux}, for which $\tau=-\nu\lambda$ is constant and $\kappa_3=\nu\tau_n$, this condition reads $\tau\neq0$ and
\[
\tau_e=\frac{\tau_n'}{\tau}-\frac{\nu\tau_n^2}{2\tau^2}+c_1e^{\nu\Theta_n/\tau}+c_2e^{2\nu\Theta_n/\tau}
\]
with constants $c_1,c_2$ and an antiderivative $\Theta_n$ of $\tau_n$, and the silhouette condition is $c_1=0$.
\end{enumerate}
By Proposition~\ref{prop:ruled}, every C-constant unit field $\eta=\lambda T+\mu N+\nu B_2$ is the unit normal along $\alpha$ of some timelike hypersurface.
\end{theorem}

\begin{proof}
(a) follows from the preceding remark and Theorem~\ref{thm:reduction}(b), and (b) from Theorem~\ref{thm:main}. (c) By Corollary~\ref{cor:asym}(a), $\mu=0$. If $\lambda=0$, then $\eta=\pm B_2$ and $\mathcal H(\eta)=0$ by Theorem~\ref{thm:reduction}(b); if $\lambda\neq0$, the assertion is Theorem~\ref{thm:typeII}. For the last statement, Proposition~\ref{prop:darboux} with $\mu=0$ gives $\lambda=-\nu\tau$, $\kappa_2=-\nu\tau_e-\frac12\tau^2$ and $\kappa_3=\nu\tau_n$, so that $\xi=\nu/\lambda=-1/\tau$ and $K_3=\nu\Theta_n$ in Theorem~\ref{thm:typeII}; substituting into \eqref{eq:typeII} and multiplying by $-\nu$ gives the formula for $\tau_e$ with $c_1=-\nu m_1$ and $c_2=-\nu m_2$.
\end{proof}

\begin{corollary}\label{cor:iso}
Let $\alpha$ be a null Cartan helix on a timelike hypersurface whose unit normal $\eta$ along $\alpha$ is C-constant.
\begin{enumerate}
\item If $\alpha$ is asymptotic, then $\eta=\lambda T+\nu B_2$ with $\nu=\pm1$, and $\alpha$ is a silhouette curve with respect to some $W\neq0$ if and only if $|\lambda|=\omega$. More precisely, $\Ax(\eta)=\R W_-$ if $\lambda=\nu\kappa_3/r$ and $\Ax(\eta)=\R W_+$ if $\lambda=-\nu\kappa_3/r$. In terms of the invariants of Proposition~\ref{prop:darboux}, $|\lambda|=|\tau|$, so the condition reads $|\tau|=\omega$, or equivalently $\tau_n^2+2\nu\tau_e\tau^2=0$.
\item If $\alpha$ is a pregeodesic, then $\eta=\lambda T+\mu N$ with $\mu=\pm1$, and $\alpha$ is a silhouette curve with respect to some $W\neq0$ if and only if $|\lambda|=r$. More precisely, $\Ax(\eta)=\R W_-$ if $\lambda=-\mu r$ and $\Ax(\eta)=\R W_+$ if $\lambda=\mu r$.
\end{enumerate}
\end{corollary}

\begin{proof}
Use Corollary~\ref{cor:asym} and Theorem~\ref{thm:main}. (a) For $\mu=0$, $\Lambda_\pm=\pm\lambda r-\nu\kappa_3$ and $\Lambda_0=\nu\kappa_3\neq0$, and $|\kappa_3|/r=\omega$. Moreover, Proposition~\ref{prop:darboux} gives $\lambda=-\nu\tau$, $\kappa_2=-\nu\tau_e-\frac12\tau^2$ and $\kappa_3=\nu\tau_n$; since $\omega^2$ is the only positive root of $y^2+2\kappa_2y-\kappa_3^2$ and $\kappa_3\neq0$, we have $|\tau|=\omega$ if and only if $\tau^4+2\kappa_2\tau^2-\kappa_3^2=0$, and the left-hand side equals $-(\tau_n^2+2\nu\tau_e\tau^2)$. (b) For $\nu=0$, $\Lambda_\pm=\mu r^2\pm\lambda r$ and $\Lambda_0=\mu\omega^2\neq0$.
\end{proof}

\begin{remark}\label{rem:general}
(a) The C-constancy of $\eta$ is what makes Theorem~\ref{thm:iso} restrictive. Indeed, let $\alpha$ be any null Cartan curve and $W$ a constant vector with $\ip{T}{W}$ nowhere zero. For any smooth $\mu,\nu$ with $\mu^2+\nu^2=1$ put $\lambda=-(\mu\ip{N}{W}+\nu\ip{B_2}{W})/\ip{T}{W}$; then $\eta=\lambda T+\mu N+\nu B_2$ satisfies $\ip{\eta}{W}=0$, and Proposition~\ref{prop:ruled} gives a timelike hypersurface along which $\alpha$ is a silhouette curve with respect to $W$.

(b) For the same reason we do not consider generalized normals $\tilde\eta=\eta+\lambda_1T+\lambda_2e$ with arbitrary functions $\lambda_1,\lambda_2$: the condition $\ip{\tilde\eta}{W}=\bar c$ is the single linear equation $\lambda_1\ip{T}{W}+\lambda_2\ip{e}{W}=\bar c-\ip{\eta}{W}$, which is solvable wherever $(\ip{T}{W},\ip{e}{W})\neq(0,0)$, so in $\E$ it imposes no condition on $\alpha$.
\end{remark}

\section{Examples}\label{sec:examples}

In Examples~\ref{ex:54}--\ref{ex:56} and~\ref{ex:58}--\ref{ex:510} the curves are given by Proposition~\ref{prop:fund} with $s_0=0$, $\alpha(0)=0$ and the initial frame
\begin{gather*}
F_0:\quad T_0=(1,1,0,0),\quad N_0=(0,0,1,0),\\
B_1^0=\tfrac12(-1,1,0,0),\quad B_2^0=(0,0,0,-1),
\end{gather*}
which has the inner products of Theorem~\ref{thm:frame}(a) and determinant one. Figures~\ref{fig:helices}--\ref{fig:hyper} use the projection $x\mapsto(x_1+x_2,x_3,x_4)$.

\begin{example}\label{ex:51}
Let $\alpha(s)=\bigl(\frac{e^s}{2}-\frac{e^{-s}}{4},\ \frac{e^s}{2}+\frac{e^{-s}}{4},\ \frac{\sin s}{\sqrt2},\ -\frac{\cos s}{\sqrt2}\bigr)$. By Theorem~\ref{thm:frame},
\begin{align*}
T&=\Bigl(\tfrac{e^s}{2}+\tfrac{e^{-s}}{4},\ \tfrac{e^s}{2}-\tfrac{e^{-s}}{4},\ \tfrac{\cos s}{\sqrt2},\ \tfrac{\sin s}{\sqrt2}\Bigr),\\
N&=\Bigl(\tfrac{e^s}{2}-\tfrac{e^{-s}}{4},\ \tfrac{e^s}{2}+\tfrac{e^{-s}}{4},\ -\tfrac{\sin s}{\sqrt2},\ \tfrac{\cos s}{\sqrt2}\Bigr),\\
B_1&=\Bigl(-\tfrac{e^s}{2}-\tfrac{e^{-s}}{4},\ -\tfrac{e^s}{2}+\tfrac{e^{-s}}{4},\ \tfrac{\cos s}{\sqrt2},\ \tfrac{\sin s}{\sqrt2}\Bigr),\\
B_2&=\Bigl(-\tfrac{e^s}{2}+\tfrac{e^{-s}}{4},\ -\tfrac{e^s}{2}-\tfrac{e^{-s}}{4},\ -\tfrac{\sin s}{\sqrt2},\ \tfrac{\cos s}{\sqrt2}\Bigr),
\end{align*}
$\kappa_2=0$ and $\kappa_3=1$, so $r=\omega=1$. By \eqref{eq:WZ}, $W_+=-2(1,1,0,0)$, $W_-=(-1,1,0,0)$ and $\Sigma=\spn\{e_3,e_4\}$. Here $\Lambda_+=\lambda+\mu-\nu$, $\Lambda_-=-\lambda+\mu-\nu$ and $\Lambda_0=\mu+\nu$. By Theorem~\ref{thm:main}, every field with $\lambda=\nu-\mu$ has the axis $W_-$ and every field with $\lambda=\mu-\nu$ has the axis $W_+$; for instance $T+N$, $\frac15(7T+3N-4B_2)$ and $-T+B_2$ have the axis $(1,1,0,0)$, and $T+B_2$ has the axis $(1,-1,0,0)$. Moreover $V_1=\frac{1}{\sqrt2}(N+B_2)=(0,0,-\sin s,\cos s)$ has $\Ax(V_1)=\spn\{e_1,e_2\}$, and $V_2=\frac{1}{\sqrt2}(N-B_2)$, which lies in $\spn\{e_1,e_2\}$, has $\Ax(V_2)=\spn\{e_3,e_4\}$. The fields $T$, $N$, $B_2$ have no axis. The curves $\mathcal C_\pm$ are shown in Fig.~\ref{fig:families} (left).
\end{example}

\begin{example}\label{ex:52}
Let $\alpha(s)=\bigl(\frac{e^s}{2}-\frac{e^{-s}}{6},\ \frac{e^s}{2}+\frac{e^{-s}}{6},\ \frac{\sin\sqrt2 s}{\sqrt6},\ -\frac{\cos\sqrt2 s}{\sqrt6}\bigr)$, so that $T=\bigl(\frac{e^s}{2}+\frac{e^{-s}}{6},\ \frac{e^s}{2}-\frac{e^{-s}}{6},\ \frac{\cos\sqrt2 s}{\sqrt3},\ \frac{\sin\sqrt2 s}{\sqrt3}\bigr)$. Then $\kappa_2=-\frac12$, $\kappa_3=\sqrt2$, $r=1$, $\omega=\sqrt2$, $W_+=-3(1,1,0,0)$, $W_-=(-1,1,0,0)$ and $\Sigma=\spn\{e_3,e_4\}$. By Corollary~\ref{cor:types}, $T+N$ and $-\sqrt2\,T+B_2$ have the axis $W_+$, $-T+N$ and $\sqrt2\,T+B_2$ have the axis $W_-$, and the type~III axis fields are $\pm V_1=\pm\frac{1}{\sqrt3}(\sqrt2N+B_2)$ and $\pm V_2=\pm\frac{1}{\sqrt3}(N-\sqrt2B_2)$.
\end{example}

\begin{example}\label{ex:53}
The curve
\[
\alpha_{2,1,1}(s)=\frac{1}{\sqrt5}\Bigl(\frac{\sinh2s}{2},\frac{\cosh2s}{2},\sin s,-\cos s\Bigr)
\]
of Proposition~\ref{prop:normalform} has $\kappa_2=\frac32$ and $\kappa_3=2$, so $r=2$, $\omega=1$, and $\ip{\alpha'''}{\alpha'''}=-3<0$; by Remark~\ref{rem:other} it admits no frame of type \eqref{eq:other}. Here $W_+=-\sqrt5(1,1,0,0)$, $W_-=\sqrt5(-1,1,0,0)$, $\Sigma=\spn\{e_3,e_4\}$, the type~I axis fields are $\pm2T+N$, the type~II axis fields are $\pm T+B_2$, and $V_1=\frac{1}{\sqrt5}(N+2B_2)$, $V_2=\frac{1}{\sqrt5}(2N-B_2)$. See Figs.~\ref{fig:helices} and~\ref{fig:families} (right).
\end{example}

\begin{figure}[H]
\centering
\includegraphics[width=0.9\linewidth]{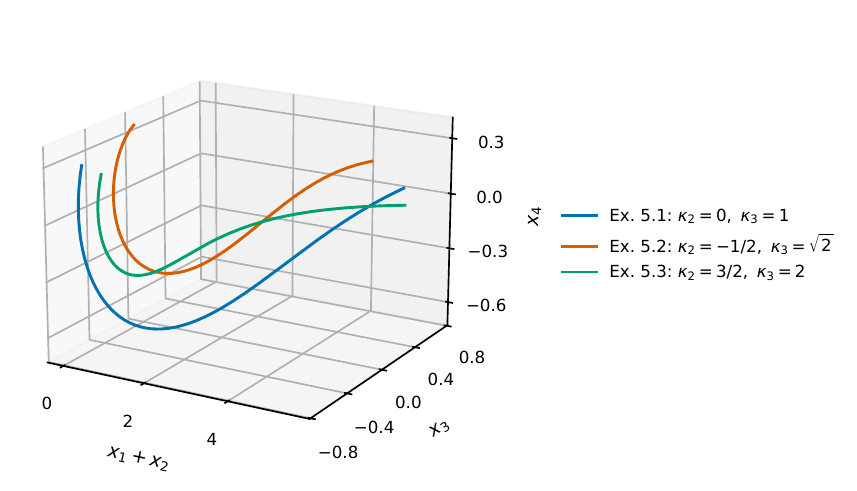}
\caption{The null Cartan helices of Examples~\ref{ex:51}--\ref{ex:53} for $s\in[-2,1.6]$.}
\label{fig:helices}
\end{figure}

\begin{figure}[H]
\centering
\includegraphics[width=\linewidth]{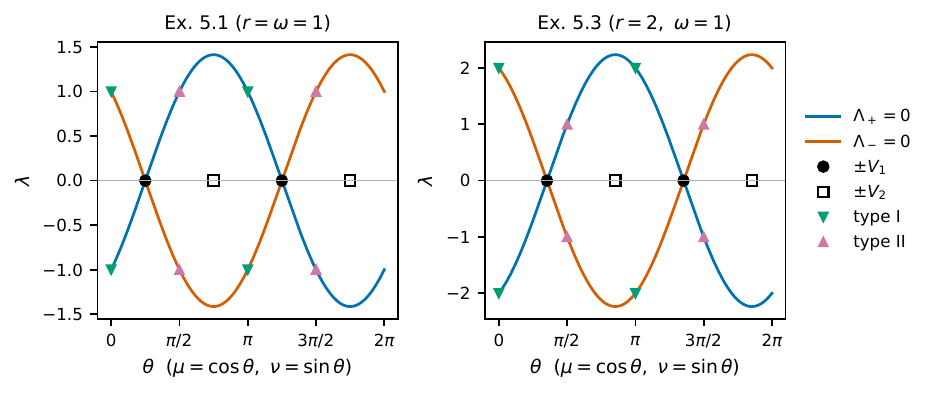}
\caption{Unit axis fields $V=\lambda T+\cos\theta\,N+\sin\theta\,B_2$ (Corollary~\ref{cor:unit}) of Example~\ref{ex:51} (left) and Example~\ref{ex:53} (right). The curves $\mathcal C_+$ ($\Lambda_+=0$) and $\mathcal C_-$ ($\Lambda_-=0$) meet at $\pm V_1$; the points $\pm V_2$ correspond to $\lambda=\Lambda_0=0$. Triangles mark the type~I ($\theta=0,\pi$) and type~II ($\theta=\pi/2,3\pi/2$) axis fields of Corollary~\ref{cor:types}.}
\label{fig:families}
\end{figure}

\begin{example}\label{ex:54}
Take $\lambda=\frac12$, $\mu=\nu=\frac{1}{\sqrt2}$ and $\kappa_3(0)=1$ in Theorem~\ref{thm:planar}: $\kappa_2=\frac14$ and $\kappa_3(s)=e^{\sqrt2 s}$. For the corresponding null curve, the field $V=\frac12T+\frac{1}{\sqrt2}(N+B_2)$ satisfies
\[
V(0)=\Bigl(\tfrac12,\tfrac12,\tfrac{1}{\sqrt2},-\tfrac{1}{\sqrt2}\Bigr),\qquad V'(0)=\Bigl(-\tfrac{\sqrt2}{8},-\tfrac{5\sqrt2}{8},\tfrac12,0\Bigr),
\]
and $\ip{V'}{V'}=\kappa_3>0$, so $\Pi=\spn\{V(0),V'(0)\}$ is spacelike and $\{V(0),V'(0)\}$ is an orthonormal basis of $\Pi$. Writing $V(s)=\cos\theta(s)\,V(0)+\sin\theta(s)\,V'(0)$ we get $\theta'=|V'|=e^{s/\sqrt2}$, i.e. $\theta(s)=\sqrt2\,(e^{s/\sqrt2}-1)$. The axes form the timelike plane
\[
\Ax(V)=\Pi^{\perp}=\spn\{(7,3,2\sqrt2,0),\ (-5,-1,0,2\sqrt2)\}.
\]
This curve is not a null Cartan helix, since $\kappa_3$ is not constant. See Fig.~\ref{fig:planar}.
\end{example}

\begin{example}\label{ex:55}
By Corollary~\ref{cor:typeI}(c) with $\lambda=1$, the null curves with $\kappa_2=0$ and $\kappa_3=(1+Ce^{2s})^{-1/2}$ are normal helices with axis field $V=T+N$ and axis
\[
W=e^{-s}\bigl(-T-N+B_1+\sqrt{1+Ce^{2s}}\,B_2\bigr);
\]
for the initial frame $F_0$ and $C>-1$ this constant vector is $W=-T_0-N_0+B_1^0+\sqrt{1+C}\,B_2^0=\bigl(-\frac32,-\frac12,-1,-\sqrt{1+C}\bigr)$. For $C=0$ we obtain a null Cartan helix congruent to that of Example~\ref{ex:51}, and for $C\neq0$ non-helical curves. See Fig.~\ref{fig:bernoulli}.
\end{example}

\begin{example}\label{ex:56}
Let $\kappa_2=0$ and $\kappa_3=-1/(1+e^s)$. By Theorem~\ref{thm:typeI} and Corollary~\ref{cor:typeI}(d), the vector
\[
W=-e^{-s}(T+N)+(1+e^{-s})(B_1-B_2)
\]
is constant; for the initial frame $F_0$ it equals $W=(-2,0,-1,2)$, a unit spacelike vector, and $\ip{T+N}{W}=1$. Moreover, the type~III field $V=\frac{1}{\sqrt2}(N-B_2)$ satisfies $\ip{V}{W}=\frac{1}{\sqrt2}$, although $\rho=\mu/\nu=-1$ solves $\kappa_3\rho^2+2\kappa_2\rho-\kappa_3=0$ for every $s$. Hence Theorem~\ref{thm:main}(a) and Corollary~\ref{cor:types}(c) do not extend to non-constant $\kappa_3$.
\end{example}

\begin{figure}[H]
\centering
\includegraphics[width=\linewidth]{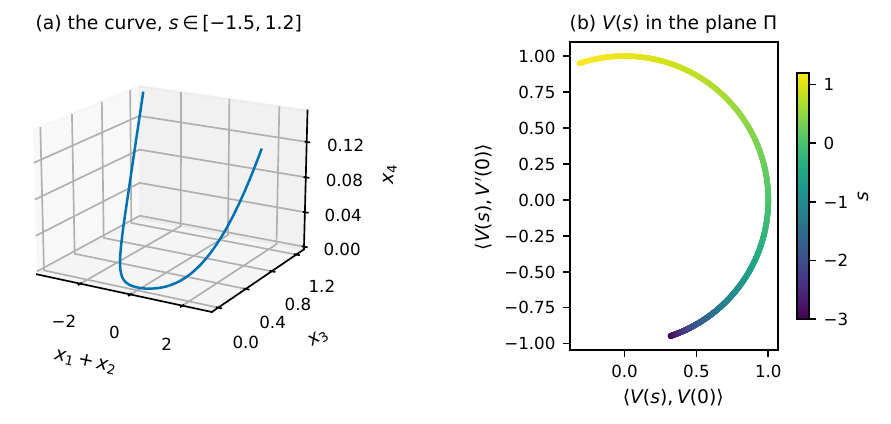}
\caption{Example~\ref{ex:54}. (a) The null curve with $\kappa_2=\frac14$, $\kappa_3=e^{\sqrt2 s}$ and initial frame $F_0$. (b) The field $V=\frac12T+\frac{1}{\sqrt2}(N+B_2)$, computed by numerical integration of \eqref{eq:cartan} for $s\in[-3,1.2]$, stays on the unit circle of the plane $\Pi=\spn\{V(0),V'(0)\}$.}
\label{fig:planar}
\end{figure}

\begin{figure}[H]
\centering
\includegraphics[width=\linewidth]{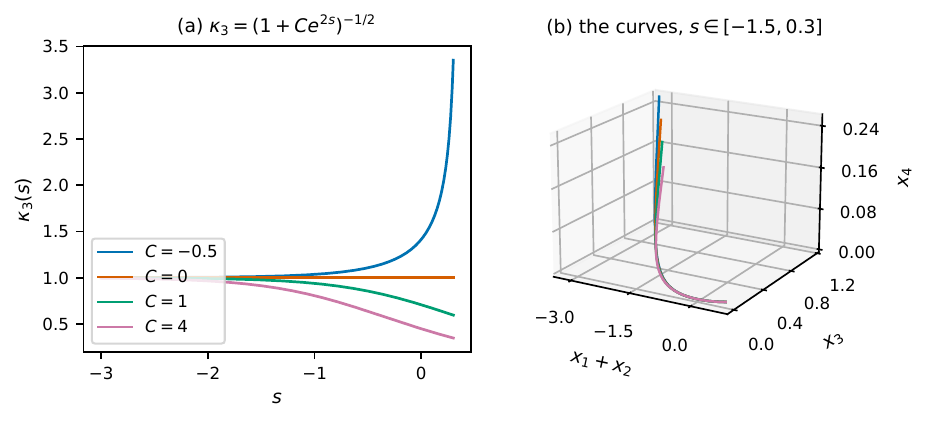}
\caption{Example~\ref{ex:55}. (a) The curvatures $\kappa_3=(1+Ce^{2s})^{-1/2}$ for $C=-0.5,0,1,4$. (b) The corresponding null curves with $\kappa_2=0$ and initial frame $F_0$; each of them is a normal helix with axis field $T+N$, and $C=0$ gives a null Cartan helix.}
\label{fig:bernoulli}
\end{figure}

\begin{example}\label{ex:57}
Let $\alpha$ be the curve of Example~\ref{ex:51} and $\eta=T+B_2$, i.e. $\lambda=1$, $\mu=0$, $\nu=1$. By Proposition~\ref{prop:ruled}, $\Phi(s,u,v)=\alpha(s)+uN(s)+v\bigl(B_1(s)-B_2(s)\bigr)$ is, near $u=v=0$, a timelike hypersurface with unit normal $\eta$ along $\alpha$. Since $\mu=0$, $\alpha$ is an asymptotic curve, and since $\lambda=\nu\kappa_3/r=1$, Corollary~\ref{cor:iso}(a) shows that $\alpha$ is a silhouette curve with respect to $W_-=(-1,1,0,0)$; indeed $\eta=\bigl(\frac{e^{-s}}{2},-\frac{e^{-s}}{2},\frac{\cos s-\sin s}{\sqrt2},\frac{\cos s+\sin s}{\sqrt2}\bigr)$ is orthogonal to $(-1,1,0,0)$. The normalized Darboux frame \eqref{eq:darbouxframe} is $e=N$, $\zeta=B_1-B_2-\frac12T$, and \eqref{eq:taus} gives $\phi=0$, $\tau_e=-\frac12$, $\tau_n=1$, $\tau=-1$, so that
\[
T'=e,\qquad \zeta'=-\tfrac12e+\eta,\qquad e'=\tfrac12T-\zeta-\eta,\qquad \eta'=-T+e .
\]
In accordance with Proposition~\ref{prop:darboux}, $\kappa_2=-\tau_e-\frac12\tau^2=0$ and $\kappa_3=\tau_n=1$, and the silhouette condition of Corollary~\ref{cor:iso}(a) takes the form $|\tau|=1=\omega$, equivalently $\tau_n^2+2\nu\tau_e\tau^2=1-1=0$. Similarly, $\eta=T+N$ ($\lambda=r$, $\mu=1$, $\nu=0$) is the unit normal along $\alpha$ of the hypersurface $\alpha(s)-uB_2(s)+v\bigl(B_1(s)-N(s)\bigr)$; then $\alpha$ is a pregeodesic and, by Corollary~\ref{cor:iso}(b), a silhouette curve with respect to $W_+$. See Fig.~\ref{fig:hyper}.
\end{example}

\begin{figure}[H]
\centering
\includegraphics[width=0.7\linewidth]{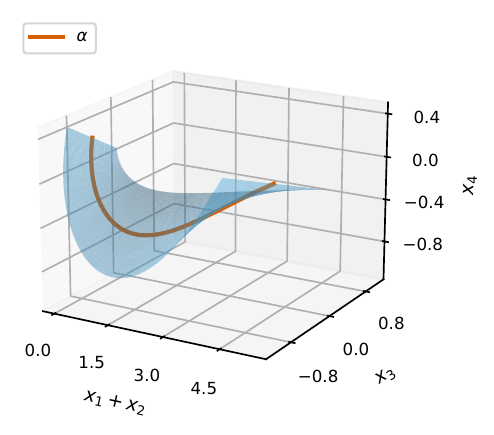}
\caption{Example~\ref{ex:57}: the null helix $\alpha$ of Example~\ref{ex:51} (bold) and the two-dimensional slice $\{\alpha(s)+uN(s):|u|\le0.6\}$ of the timelike hypersurface $\Phi(s,u,v)=\alpha(s)+uN(s)+v(B_1(s)-B_2(s))$, for $s\in[-2,1.2]$. Along $\alpha$ the unit normal of $\Phi$ is $\eta=T+B_2$, and $\alpha$ is an asymptotic silhouette curve with respect to $W_-=(-1,1,0,0)$.}
\label{fig:hyper}
\end{figure}

\begin{example}\label{ex:58}
Let $\kappa_2=1$, $\kappa_3=1+s^2$ and let $V=\frac{1}{\sqrt2}(N+B_2)$, a field of type~III. In Theorem~\ref{thm:munu}, $g_0=-\sqrt2\,s/(1+s^2)$ and $g_1=\sqrt2/(1+s^2)$, so $u=-g_0/g_1=s$ and $\mu(u'+u^2)+\lambda u-\nu\kappa_3=(1+s^2-\kappa_3)/\sqrt2=0$. Hence $\Ax(V)$ is spanned by the constant vector $W$ with $c_W=e^{s^2/2}$, which by \eqref{eq:Wc} is
\[
W=e^{s^2/2}\bigl(-s^2T+sN+B_1-sB_2\bigr);
\]
for the initial frame $F_0$ it equals $B_1^0=\frac12(-1,1,0,0)$, so $V(s)$ lies in the degenerate hyperplane $x_1+x_2=0$. The second solution $f=e^{s^2/2}\int_0^se^{-t^2}dt$ of $\mathcal P_V[f]=0$ gives $g_0f+g_1f'=\sqrt2\,e^{-s^2/2}/(1+s^2)$, which is not constant; hence $\mathcal H(V)=\Ax(V)=\R W$. Here $\rho=\mu/\nu=1$ does not satisfy $\kappa_3\rho^2+2\kappa_2\rho-\kappa_3=0$, so the characterization of type~III axis fields in Corollary~\ref{cor:types}(c) fails for non-constant curvatures in this direction as well (compare Example~\ref{ex:56}).
\end{example}

\begin{example}\label{ex:59}
Let $\kappa_3=1$ and $V=T+B_2$, so that $\lambda=\nu=\xi=1$ and $K_3=s$ in Theorem~\ref{thm:typeII}. The curve is a normal helix with axis field $V$ if and only if $\kappa_2=m_1e^{-s}+m_2e^{-2s}$ for constants $m_1,m_2$; then
\begin{gather*}
W=(m_1+m_2e^{-s}-e^s)T+e^s(N+B_1-B_2)+m_1B_2,\\
\ip{V}{W}=m_1 .
\end{gather*}
For $m_1=m_2=0$ we obtain a null Cartan helix congruent to that of Example~\ref{ex:51}, and $W=W_-$. For the initial frame $F_0$, $W=\bigl(m_1+m_2-\frac32,\ m_1+m_2-\frac12,\ 1,\ 1-m_1\bigr)$ and $\ip{W}{W}=m_1^2+2m_2$. By Theorem~\ref{thm:iso}(c) and Proposition~\ref{prop:ruled}, for $m_1\neq0$ these curves are asymptotic isophotic curves, but not silhouette curves, of timelike hypersurfaces with unit normal $T+B_2$ along the curve.
\end{example}

\begin{example}\label{ex:510}
Let $\lambda=2$, $\kappa_3=e^s$ and $\kappa_2=2-\frac12e^{2s}$, i.e. $q=1$ in Corollary~\ref{cor:typeI}(e), and let $V=2T+N$. Then $E_1=E_2=0$, so all vectors $W_{c_0,K}$ of Theorem~\ref{thm:typeI} are constant. For the initial frame $F_0$,
\[
W_{2,0}=(1,2,0,1),\qquad W_{0,1}=(-3,-2,-2,-1),
\]
so that $\ip{V(s)}{(1,2,0,1)}=2$ and $\ip{V(s)}{(-3,-2,-2,-1)}=0$ for all $s$: $V(s)$ moves in a fixed affine plane which does not pass through the origin, and $\Ax(V)=\R(-3,-2,-2,-1)$ is a null line. Moreover $W_{2,0}=\kappa_2T+B_1-e^sB_2$ is the vector $W_1$ of Remark~\ref{rem:ktype}, so that $\ip{T(s)}{(1,2,0,1)}=1$ and $\ip{N(s)}{(1,2,0,1)}=0$ for all $s$.
\end{example}

\section{Conclusions}

We have studied null Cartan normal helices in Minkowski space-time with respect to the standard Cartan frame, which exists and is unique for every null Cartan curve. For arbitrary $\kappa_2$ and nowhere vanishing $\kappa_3$, the constant vectors $W$ for which $\ip{V}{W}$ is constant correspond to the common solutions of two linear differential equations, of order four and of order at most two (Theorem~\ref{thm:reduction}); this bounds the dimension of the space of such vectors by two, shows that $B_2$ is never an axis field (both statements fail for $\kappa_3\equiv0$), and yields the normal helices with a prescribed axis field and a prescribed third curvature, away from the zeros of $\ip{W}{T}$, by two quadratures (Theorem~\ref{thm:construct}), explicitly for fields of type~II (Theorem~\ref{thm:typeII}). The $k$-type null slant helices of \cite{nes15} with $k=0,1,3$ are the special cases $V=T,N,B_2$. For null Cartan helices the problem reduces to common roots of two characteristic polynomials (Theorem~\ref{thm:main}): a C-constant normal field $V=\lambda T+\mu N+\nu B_2$ admits an axis if and only if $\Lambda_+=0$, $\Lambda_-=0$ or $\lambda=\Lambda_0=0$; the constant $\ip{V}{W}$ is then zero; and the constant vectors $W$ with $\ip{V}{W}=0$ form the span of one or both of the invariant null lines $\R W_\pm$ or the invariant spacelike plane $\Sigma$ of the helix. For fields with $\mu\nu\neq0$ the existence of an axis with $\ip{V}{W}=0$ reduces to a Riccati equation (Theorem~\ref{thm:munu}), and the constant vectors $W$ with $\ip{V}{W}=0$ form a two-dimensional space exactly when $V$ moves in a fixed plane through the origin (Theorem~\ref{thm:planar}); for fields of type~I a Bernoulli equation appears, nonzero constants $\ip{V}{W}$ occur, and there are non-helical curves with two linearly independent axes (Corollary~\ref{cor:typeI}), which have no counterpart in $\mathbb{E}^3_1$. On timelike hypersurfaces we introduced a normalized Darboux frame with four invariants, showed that the normal is C-constant exactly when two of them are constant, and described the null Cartan curves that are isophotic or silhouette curves with C-constant normal in terms of the surface invariants (Proposition~\ref{prop:darboux}, Theorem~\ref{thm:iso} and Corollary~\ref{cor:iso}).

The methods are elementary, and for arbitrary curvatures they give closed formulas only for fields of type~II and for $T$; we regard the unifying framework, which treats all C-constant normal fields and all curvatures by one pair of linear equations, and the phenomena that distinguish $\E$ from $\mathbb{E}^3_1$ as the main contributions of the paper. The results also have a physical reading: for a null Cartan helix the axis spaces are the principal null directions and the plane of rotation of the generator of the helix, whereas along the trajectories of the particle model of \cite{nersessian}, which have $\kappa_3\equiv0$, the conserved momentum and the Pauli--Lubanski vector belong to $\mathcal H(V)$ for every C-constant normal field $V$, and $\lambda$ and $\nu$ are, up to constant factors, the corresponding constants $\ip{V}{W}$.

Natural further problems are a geometric description of the curves along which a field with $\mu\nu\neq0$ moves in an affine plane not through the origin, normal helices whose axis field has non-constant components, and the analogous questions for null Cartan curves in Minkowski spaces of higher dimension and for pseudo-null and partially null curves.

\subsection*{Code availability}
The Python package \texttt{nullhelix} \cite{code} verifies every lemma, proposition, theorem, corollary and example of this paper symbolically (with SymPy) or numerically (with SciPy) and generates Figures~\ref{fig:helices}--\ref{fig:hyper}. Each check is labelled by the number of the corresponding statement.

\section*{Statements and Declarations}

\begin{itemize}
\item \textbf{Funding:} No funding was received for conducting this study.
\item \textbf{Competing interests:} The authors declare that they have no competing interests.
\item \textbf{Ethics approval and consent to participate:} Not applicable.
\item \textbf{Consent for publication:} Not applicable.
\item \textbf{Data availability:} No new data were generated or analyzed in support of this research. The Python package \texttt{nullhelix}, which verifies every lemma, proposition, theorem, corollary and example of this article symbolically (with SymPy) or numerically (with SciPy) and generates Figs.~\ref{fig:helices}--\ref{fig:hyper}, is openly available in Zenodo at \url{https://doi.org/10.5281/zenodo.23672878} \cite{code}. Each check in the package is labeled by the number of the corresponding statement.
\item \textbf{Materials availability:} Not applicable.
\item \textbf{Code availability:} The Python scripts used for all computations and figures in this article, together with their output, are provided at Zenodo \cite{code}.
\item \textbf{Author contribution:} All authors contributed equally to the writing of this paper.
\item \textbf{Use of AI tools:} During the preparation of this work, the authors used Claude (Anthropic) in order to assist in writing Python scripts for the symbolic and numerical verification of the identities, numerical values, tables, and figures in the paper. All mathematical results and proofs are the authors' own, and the output of the scripts was checked by the authors. After using this tool, the authors reviewed and edited the content as needed and take full responsibility for the content of the publication.
\end{itemize}

\end{document}